\documentclass[a4paper,10pt,fleqn]{article}
\usepackage{amsmath,makeidx,amsfonts,a4wide,amssymb,amsthm,longtable}
\newtheorem{theorem}{Theorem}

\numberwithin{equation}{section}
\numberwithin{lemma}{section}
\numberwithin{theorem}{section}
\numberwithin{corollary}{section}

\allowdisplaybreaks

\begin{document}
\title{A study on Horn matrix functions and its confluent cases}
\author{Ravi Dwivedi\footnote{Department of Science (Mathematics), Govt. Naveen College Bhairamgarh, Bijapur (CG) 494450, India. \newline Email: dwivedir999@gmail.com} 
}
\maketitle
\begin{abstract}
In this paper, we give the matrix version of Horn's hypergeometric function and its confluent cases. We also discuss the regions of convergence, the system of matrix differential equations of bilateral type, differential formulae and infinite summation formulae satisfied by these hypergeometric matrix functions. We also give the certain integral representation of these hypergeometric matrix functions. The study of these 23 matrix functions leads to completing the matrix generalization of Horn's list of 34 hypergeometric series.


\end{abstract}

\section{Introduction}\label{s:i}
In general, Horn defined that a double power series $\sum_{m,n = 0}^{\infty} \mathcal{A}_{m,n} x^m \, y^n$ is called a hypergeometric series if the two quotients $\frac{\mathcal{A}_{m+1,n}}{\mathcal{A}_{m,n}} = f(m,n)$ and $\frac{\mathcal{A}_{m,n+1}}{\mathcal{A}_{m,n}} = g(m,n)$ are rational functions of $m$ and $n$. Horn puts $f(m,n) = \frac{F(m,n)}{F'(m,n)}$ and $g(m,n) = \frac{G(m,n)}{G'(m,n)}$, where $F$, $F'$, $G$, $G'$ are polynomials in $m$, $n$ of respective degrees $p$, $p'$, $q$, $q'$. The highest of the four numbers $p$, $p'$, $q$, $q'$ is the order of the hypergeometric series. Horn investigated a complete list of  $34$ distinct convergent series of order $2$, out of these 34 distinct series 14 are complete series for which $p = p' = q = q' = 2$ and rest of the 20 series are confluent cases of the 14 complete series, for more detail see \cite{emo}, \cite{jh}, \cite{ph}, \cite{sk}.  Recently,  Brychkova and Savischenko 
studied various properties of Horn's functions and confluent form of Horn's functions. The intimate relationship between Horn functions and some fundamental equations of mathematical physics shows the importance of these special functions. Horn functions arise by partial separation of a canonical system of partial differential equations and by some consequence it's shown that these functions appear as solution of the $4$-variable wave equation, $3$-variable wave and heat equations and $2$-variable Helmholtz equation, \cite{kmm1}, \cite{kmm2}, \cite{kmm3}.

 The matrix generalization of special function is being initiated by J\'odar and Cort\'es and studied the gamma matrix function, beta matrix function and Gauss hypergeometric matrix function \cite{jjc98a}, \cite{jjc98b}. The matrix analogue of Appell functions and Lauricella functions of several variable have been studied in \cite{al}, \cite{RD}, \cite{RD2}, \cite{RD3}.  The confluent cases of Appell matrix functions are  given in \cite{cdss}. In this paper, we study the matrix analogue of remaining Horn functions and their confluent cases. We give the regions of convergence, differential formulae, infinite summation formulae and system of bilateral type matrix differential equations obeyed by these matrix functions.  
The section-wise treatment is as follows. 

In Section~2, we list the basic definitions and results that are needed in the sequel. In Section~3, we define the Horn matrix function and Horn confluent matrix functions. We also give here the regions of convergence and system of bilateral type matrix differential equations obeyed by these matrix functions. We also give here the certain integral representations of these hypergeometric matrix functions.  In Section~4, We obtain the differential formulae satisfied by Horn matrix functions and Horn confluent matrix functions. Finally, in Section~5, the infinite summation formulae for Horn matrix functions and their confluent matrix functions are presented.
	
\section{Preliminaries} 
Let $\mathbb{C}^{r \times r}$ denote the vector space of all $r$-square matrices with complex entries.	For  $A\in \mathbb{C}^{r\times r}$, $\sigma(A)$ is the spectrum of A. The spectral abscissa of $A$ is given by $\alpha(A) = \max\{\,\Re(z) \mid z\in \sigma(A)\,\}$,  where $\Re(z)$  denotes the real  part of a complex number $z$. If $\beta(A) = \min\{\,\Re(z) \mid z \in \sigma(A)\,\}$, then $\beta(A) = -\alpha(-A)$. A square matrix $A$ is said to be positive stable if  $\beta(A) >0$. The 2-norm of $A$ is denoted by $\Vert A\Vert$ and defined by 
\begin{equation}
\Vert A\Vert = \max_{x\neq 0} \frac{\Vert {Ax}\Vert_2}{ \Vert x \Vert_2} = \max\{ \,\sqrt{\lambda} \mid  \lambda\in \sigma(A^*A)\,\},\label{c1eq1.1}
\end{equation}
where for any vector $x$ in the $r$-dimensional complex space,   $\Vert{x}\Vert_2 = (x^*x)^\frac{1}{2}$ is the Euclidean norm of $x$ and $A^*$ denotes the transposed conjugate of $A$. If $f(z)$ and $g(z)$ are holomorphic functions of the complex variable $z$, which are defined in an open set $\Omega$ of the complex plane, and $A$ is a matrix in $\mathbb{C}^{r \times r}$ with  $\sigma(A)\subset\Omega$,  then from the properties of the matrix functional calculus \cite{ds57}, it follows that 
\begin{equation}
f(A)g(A) = g(A)f(A).
\end{equation}
Furthermore, if $B\in \mathbb{C}^{r \times r}$ is a matrix for which $\sigma(B)\subset\Omega$, and if $AB = BA$, then
\begin{equation}
f(A)g(B) = g(B)f(A).
\end{equation}
The reciprocal gamma function $\Gamma^{-1}(z)=1/\Gamma(z)$ is an entire function of the complex variable $z$. The image of $\Gamma^{-1}(z)$ acting on $A$, denoted by $\Gamma^{-1}(A)$, is a well defined matrix. If $A+nI$ is invertible for all integers $n\geq 0$, then the reciprocal gamma function is defined as \cite{jjc98a}
	  \begin{equation}
	 \Gamma^{-1}(A)= A(A+I)\dots (A+(n-1)I)\Gamma^{-1}(A+nI) , \  n\geq 1.\label{eq.07}
	  \end{equation}
	  The Pochhammer symbol $(z)_n$, $z\in \mathbb{C}$, is defined as
	  \begin{equation}
	  (z)_n = \begin{cases}
	  1, & \text{if $n = 0$}\\
	  z(z+1) \dots (z+n-1), & \text{if $n\geq 1$}.
	  \end{cases}
	  \end{equation}
	  By  application of the matrix functional calculus, the Pochhammer symbol  for  $A\in \mathbb{C}^{r\times r}$ is given by
	  \begin{equation}
	  (A)_n = \begin{cases}
	  I, & \text{if $n = 0$}\\
	  A(A+I) \dots (A+(n-1)I), & \text{if $n\geq 1$}.
	  \end{cases}\label{c1eq.09}
	  \end{equation}
	  This gives
	  	   \begin{equation}
	   (A)_n = \Gamma^{-1}(A) \ \Gamma (A+nI), \qquad n\geq 1.\label{c1eq.010}
	   \end{equation} 

If $A\in \mathbb{C}^{r \times r}$ is such that $\Re(z)>0$ for all eigenvalues $z$ of $A$, then $\Gamma(A)$ can be expressed as \cite{jjc98a}
\[ \Gamma(A) = \int_{0}^{\infty} e^{-t} \, t^{A-I}\, dt.
\]

\section{Horn matrix functions}
Horn list of fourteen complete series contains four Appell series of two variables and remaining ten $G$ and $H$-hypergeometric series. The matrix version of Appell series is given in \cite{al}, \cite{RD}, \cite{RD3}. We give here the definition of remaining ten $G$ and $H$-hypergeometric with matrix parameters. Let $A$, $A'$, $B$, $B'$, $C$, $C'$ and $C''$ be matrices in $\mathbb{C}^{r\times r}$ such that $C+kI$, $C'+kI$ and $C''+kI$ are invertible for all integers $k\ge 0$. Then, we define
\begin{align}
G_1(A, B, B'; x, y) &= \sum_{m,n =0}^{\infty} (A)_{m+n} (B)_{n-m} (B')_{m-n} \frac{x^m \, y^n}{m!\, n!}, \label{3.1}
\end{align}
\begin{align}
G_2(A, A', B, B'; x, y) &= \sum_{m,n =0}^{\infty} (A)_{m} (A')_{n} (B)_{n-m} (B')_{m-n} \frac{x^m \, y^n}{m!\, n!}, \label{3.2}
\end{align}
\begin{align}
G_3(A, A'; x, y) &= \sum_{m,n =0}^{\infty} (A)_{2n-m} (A')_{2m-n} \frac{x^m \, y^n}{m!\, n!}, \label{3.3}
\end{align}
\begin{align}
H_1(A, B, C, C'; x, y) &= \sum_{m,n =0}^{\infty} (A)_{m-n} (B)_{m+n} (C)_{n} (C')_{m}^{-1} \frac{x^m \, y^n}{m!\, n!}, \label{3.4}
\end{align}
\begin{align}
H_2(A, B, C, C'; C'' x, y) &= \sum_{m,n =0}^{\infty} (A)_{m-n} (B)_{m} (C)_{n} (C')_{n} (C'')_{m}^{-1} \frac{x^m \, y^n}{m!\, n!}, \label{3.5}
\end{align}
\begin{align}
H_3(A, B; C; x, y) &= \sum_{m,n =0}^{\infty} (A)_{2m+n} (B)_{n} (C)_{m+n}^{-1} \frac{x^m \, y^n}{m!\, n!}, \label{3.6}
\\[5pt]
H_4(A, B; C, C'; x, y) &= \sum_{m,n =0}^{\infty} (A)_{2m+n} (B)_{n} (C)_{m}^{-1} (C')_{n}^{-1} \frac{x^m \, y^n}{m!\, n!}, \label{3.7}
\end{align}
\begin{align}
H_5(A, B; C; x, y) &= \sum_{m,n =0}^{\infty} (A)_{2m+n} (B)_{n-m} (C)_{n}^{-1} \frac{x^m \, y^n}{m!\, n!}, \label{3.8}
\end{align}
\begin{align}
H_6(A, B; C; x, y) &= \sum_{m,n =0}^{\infty} (A)_{2m-n} (B)_{n-m} (C)_{n} \frac{x^m \, y^n}{m!\, n!}, \label{3.9}
\end{align}
\begin{align}
H_7(A, B; C, C'; x, y) &= \sum_{m,n =0}^{\infty} (A)_{2m-n} (B)_{n} (C)_{n} (C')_{m}^{-1} \frac{x^m \, y^n}{m!\, n!}. \label{3.10}
\end{align}
There are twenty confluent functions of two variable hypergeometric functions among of them seven are confluent cases of Appell functions known as Humbert functions. The matrix analogue of these seven Humbert functions  have been studied in, \cite{cdss}. The remaining 13 confluent hypergeometric functions, obtained as limiting cases of Horn functions, has been listed fairly in \cite{sk}. Now, we define the matrix analogue of these 13 confluent hypergeometric functions. 

\begin{align}
	\Gamma_1 (A, B, B'; x, y) = \sum_{m, n \ge 0} (A)_m (B)_{n-m} (B')_{m-n} \frac{x^m y^n}{m! \, n!};\label{5.1}
\end{align}
\begin{align}
	\Gamma_2 (B, B'; x, y) = \sum_{m, n \ge 0} (B)_{n-m} (B')_{m-n} \frac{x^m y^n}{m! \, n!};\label{5.2}
\end{align}
\begin{align}
	\mathcal{H}_1 (A, B, C; x, y) = \sum_{m, n \ge 0} (A)_{m-n} (B)_{m+n} (C)_{m}^{-1} \frac{x^m y^n}{m! \, n!};\label{5.3}
\end{align}
\begin{align}
	\mathcal{H}_2 (A, B, B'; C; x, y) = \sum_{m, n \ge 0} (A)_{m-n} (B)_{m} (B')_{n} (C)_m^{-1} \frac{x^m y^n}{m! \, n!};\label{5.4}
\end{align}
\begin{align}
	\mathcal{H}_3 (A, B; C; x, y) = \sum_{m, n \ge 0} (A)_{m-n} (B)_{m} (C)_m^{-1} \frac{x^m y^n}{m! \, n!};\label{5.5}
\end{align}
\begin{align}
	\mathcal{H}_4 (A, B'; C; x, y) = \sum_{m, n \ge 0} (A)_{m-n} (B')_{n} (C)_m^{-1} \frac{x^m y^n}{m! \, n!};\label{5.6}
\end{align}
\begin{align}
	\mathcal{H}_5 (A; C; x, y) = \sum_{m, n \ge 0} (A)_{m-n} \, (C)_m^{-1} \frac{x^m y^n}{m! \, n!};\label{5.7}
\end{align}
\begin{align}
	\mathcal{H}_6 (A; C; x, y) = \sum_{m, n \ge 0} (A)_{2m+n} \, (C)_{m+n}^{-1} \, \frac{x^m y^n}{m! \, n!};\label{5.8}
\end{align}
\begin{align}
	\mathcal{H}_7 (A; C, C'; x, y) = \sum_{m, n \ge 0} (A)_{2m+n} (C)^{-1}_{m}  (C')_n^{-1} \frac{x^m y^n}{m! \, n!};\label{5.9}
\end{align}
\begin{align}
	\mathcal{H}_8 (A, B; x, y) = \sum_{m, n \ge 0} (A)_{2m-n} (B)_{n-m}  \frac{x^m y^n}{m! \, n!};\label{5.10}
\end{align}
\begin{align}
	\mathcal{H}_9 (A, B; C; x, y) = \sum_{m, n \ge 0} (A)_{2m-n} \, (B)_n \, (C)_{m}^{-1} \, \frac{x^m y^n}{m! \, n!};\label{5.11}
\end{align}
\begin{align}
	\mathcal{H}_{10} (A; C; x, y) = \sum_{m, n \ge 0} (A)_{2m-n} (C)^{-1}_{m}  \frac{x^m y^n}{m! \, n!};\label{5.12}
\end{align}
\begin{align}
	\mathcal{H}_{11} (A, B, C; C'; x, y) = \sum_{m, n \ge 0} (A)_{m-n} (B)_{n} (C)_n (C')_m^{-1}  \frac{x^m y^n}{m! \, n!};\label{5.13}.
\end{align}
It can be verified, using $\lim_{\varepsilon \to 0} \, \left(\frac{1}{\varepsilon}I\right)_m \ \varepsilon^m = I$, that the matrix functions defined in \eqref{5.1}-\eqref{5.13} are confluent cases of Horn matrix  functions. Indeed, we have
\begin{align}
	\Gamma_1 (A, B; B'; x, y) &= \lim_{\varepsilon \to 0} G_2\left(A, \frac{1}{\varepsilon}I, B, B'; x, \varepsilon y\right);\label{eq1.13}
\end{align}
\begin{align}
	\Gamma_2 (B, B'; x, y)  &= \lim_{\varepsilon \to 0} G_1\left(\frac{1}{\varepsilon}I, B, B'; \varepsilon x, \varepsilon y\right);\label{eq1.14}
\end{align}
\begin{align}
	\mathcal{H}_1 (A, B; C'; x, y) &= \lim_{\varepsilon \to 0} H_1 \left(A, B, \frac{1}{\varepsilon}I; C'; x, \varepsilon y\right);\label{eq1.15}
\end{align}
\begin{align}
	\mathcal{H}_2 (A, B, C; C''; x, y)  &=\lim_{\varepsilon \to 0} H_2\left(A, B, C, \frac{1}{\varepsilon}I;  C''; x, \varepsilon y\right);
	\\[5pt]
	\mathcal{H}_3 (A, B, C''; x, y)  &=\lim_{\varepsilon \to 0} H_2\left(A, B, \frac{1}{\varepsilon}I, \frac{1}{\varepsilon}I;  C''; x, (\varepsilon)^2 y\right)
	\\[5pt]
	&=\lim_{\varepsilon \to 0} \mathcal{H}_2 (A, B, \frac{1}{\varepsilon}I; C''; x, \varepsilon y);
	\\[5pt]
	\mathcal{H}_4 (A, C, C''; x, y)  &=\lim_{\varepsilon \to 0} H_2\left(A, \frac{1}{\varepsilon}I, C, \frac{1}{\varepsilon}I;  C''; \varepsilon x, \varepsilon y\right)
	\\[5pt]
	&=\lim_{\varepsilon \to 0} \mathcal{H}_2 (A, \frac{1}{\varepsilon}I, C; C''; \varepsilon x,  y);
	\\[5pt]
	\mathcal{H}_5 (A; C''; x, y)  &=\lim_{\varepsilon \to 0} H_2\left(A, \frac{1}{\varepsilon}I, \frac{1}{\varepsilon}I, \frac{1}{\varepsilon}I;  C''; \varepsilon x, (\varepsilon)^2 y\right)
	\\[5pt]
	&=\lim_{\varepsilon \to 0} \mathcal{H}_2 (A, \frac{1}{\varepsilon}I, \frac{1}{\varepsilon}I; C''; \varepsilon x,  \varepsilon y);
	\\[5pt]
	\mathcal{H}_6 (A; C; x, y) & =\lim_{\varepsilon \to 0} H_3 \left(A, \frac{1}{\varepsilon}I; C;  x, \varepsilon y\right);
	\\[5pt]
	\mathcal{H}_7 (A; C, C'; x, y) & =\lim_{\varepsilon \to 0} H_4 \left(A, \frac{1}{\varepsilon}I; C, C';  x, \varepsilon y\right);
	\\[5pt]
		\mathcal{H}_8 (A, B; x, y) & =\lim_{\varepsilon \to 0} H_6 \left(A, B, \frac{1}{\varepsilon}I;  x, \varepsilon y\right);
	\\[5pt]
	\mathcal{H}_9 (A, B; C'; x, y) & =\lim_{\varepsilon \to 0} H_7 \left(A, B, \frac{1}{\varepsilon}I; C';  x, \varepsilon y\right);
	\\[5pt]
	\mathcal{H}_{10} (A; C'; x, y) & =\lim_{\varepsilon \to 0} H_7 \left(A, \frac{1}{\varepsilon}I, \frac{1}{\varepsilon}I; C';  x, (\varepsilon)^2 y\right);
	\\[5pt]
	\mathcal{H}_{11} (A, C, C'; C''; x, y) & =\lim_{\varepsilon \to 0} H_2 \left(A, \frac{1}{\varepsilon}I, C, C'; C'';  \varepsilon x,  y\right).
\end{align}

\subsection{Regions of convergence}
We now determine the convergence of these matrix functions. To obtain so, we extend the well known technique develop by Horn given in \cite{sk}. Consider the hypergeometric matrix series
\begin{align}
F(x, y) = \sum_{m,n = 0}^{\infty} \mathcal{C}_{m,n} \, x^m y^n,
\end{align}
which gives
\begin{align}
\Vert F(x, y)\Vert & \le \sum_{m,n = 0}^{\infty} \Vert \mathcal{C}_{m,n}\Vert \, \vert x\vert^m \, \vert y\vert^n,\nonumber\\
& =  \sum_{m,n = 0}^{\infty}  \mathcal{A}_{m,n} \, \vert x\vert^m \, \vert y\vert^n.
\end{align}
Define
\begin{align}
\rho (m,n) = \left\vert \lim_{u \to \infty} f(mu, nu)\right\vert^{-1}, \quad m > 0, n \ge 0,
\\[5pt]
\sigma (m,n) = \left\vert \lim_{u \to \infty} g(mu, nu)\right\vert^{-1}, \quad m\ge 0, n > 0,
\end{align}
where $f(m, n) = \frac{\mathcal{A}_{m+1, n}}{\mathcal{A}_{m,n}}$ and $g(m, n) = \frac{\mathcal{A}_{m, n+1}}{\mathcal{A}_{m,n}}$. Now one can proceed in the same way as by Horn to find the region of convergence for Horn matrix functions. We will start by finding the region of convergence of Horn matrix function $G_1$ defined in \eqref{3.1}.
\begin{theorem}
	Let $A$, $B$ and $B'$ be matrices in $\mathbb{C}^{r\times r}$. Then the matrix function $G_1$ defined in \eqref{3.1} converges absolutely for $r + s < 1$, $\vert x\vert \le r$, $\vert y\vert \le s$.
\end{theorem}
\begin{proof}
	Consider the matrix series 
	\begin{align}
	G_1(A, B, B'; x, y) = \sum_{m,n =0}^{\infty} (A)_{m+n} (B)_{n-m} (B')_{m-n} \frac{x^m \, y^n}{m!\, n!}.
	\end{align}
	This implies
	\begin{align}
\Vert G_1(A, B, B'; x, y)\Vert &\le \sum_{m,n =0}^{\infty} (\Vert A\Vert)_{m+n} (\Vert B\Vert)_{n-m} (\Vert B'\Vert)_{m-n} \frac{\vert x\vert^m \, \vert y\vert^n}{m!\, n!}\nonumber\\
& = \sum_{m,n =0}^{\infty} \mathcal{A}_{m,n} \vert x\vert^m \, \vert y\vert^n. 
\end{align}	
So, we have 
\begin{align}
f_{m,n} = \frac{\mathcal{A}_{m+1, n}}{\mathcal{A}_{m,n}} = \frac{(\Vert A\Vert + m + n) \, (\Vert B'\Vert + m - n)}{(\Vert B\Vert + n - m -1) \, (m+1)},
\end{align}
\begin{align}
g_{m,n} = \frac{\mathcal{A}_{m, n+1}}{\mathcal{A}_{m,n}} = \frac{(\Vert A\Vert + m + n) \, (\Vert B\Vert + m - n)}{(\Vert B'\Vert + m - n -1) \, (n+1)} 
\end{align}
and 
\begin{align}
 \rho (m,n) = \left\vert \lim_{u \to \infty} f(mu, nu)\right\vert^{-1} = \frac{n}{m + n},
 \end{align}
 \begin{align}
 \sigma (m,n) = \left\vert \lim_{u \to \infty} g(mu, nu)\right\vert^{-1} = \frac{m}{m+n}.
\end{align}
Therefore the region of convergence is given by 
\begin{align}
\mathfrak{C} = \{(r,s) \mid 0 < r < \rho_{1,0} \cap 0 < s < \sigma_{0,1}\} = K[1, 1],
\end{align}
\begin{align}
\mathfrak{Z} = \left\{(r,s) \mid \forall (m,n) \in \mathbb{R}_+^2: 0 < r < \frac{n}{m+n} \cup 0 < s < \frac{m}{m+n}\right\}.\label{3.22}
\end{align}
Eliminating $m$ and $n$ from \eqref{3.22} gives required  region of absolute convergence. 
\end{proof}
Note that the region of absolute convergence of Horn matrix series $G_1$ is identical with the region of convergence in the complex case.
 The region of convergence of other Horn's matrix series is same as of the series with complex parameters, one can see \cite{emo}, \cite{sk}.

 Jod\'ar and Cort\'es \cite{jc00}, introduced the concept of a fundamental set of solutions for matrix  differential equations of the type
 \begin{align}
 X'' = f_1(z) X' + f_2(z) X f_3(z) + X' f_4(z),
 \end{align} 
 where $f_i, 1 \le i \le 4$ are matrix valued functions of complex variable $z$. A  closed form general solution of such bilateral type matrix differential equation is determined  in terms of Gauss hypergeometric matrix function. In \cite{al}, systems of bilateral type matrix differential equation have been given for Appell matrix functions of two variables. We now give the systems of matrix differential equations of bilateral type obeyed by the Horn matrix functions  and Horn confluent matrix functions defined in \eqref{3.1}-\eqref{5.13}. Let $U_{xx} = \frac{\partial^2 U}{\partial x^2}$, $U_{xy} = \frac{\partial^2U}{\partial x \partial y}$, $U_{yy} = \frac{\partial^2 U}{\partial y^2}$, $U_{x} = \frac{\partial U}{\partial x}$, $U_{y} = \frac{\partial U}{\partial y}$. Then the system of matrix differential equations of bilateral type obeyed by the Horn matrix function $G_1$ is given below:
\begin{theorem}\label{t3.2}
Let $A$, $B$ and $B'$ be matrices in $\mathbb{C}^{r\times r}$ such that $BB' = B'B$. Then the system of matrix differential equations of bilateral type satisfied by the Horn matrix function $G_1$ is given by
\begin{align}
&x (1 + x) U_{xx} - y U_{xy} - y^2 U_{yy} + U_x(I-B) + x (A+I) U_x + x U_x B' + yU_y B'\nonumber\\
& \quad  - y(A+I) U_y + AUB' = 0,\label{3.23}
\\[5pt]
& y (1 + y) U_{yy} - x U_{xy} - x^2 U_{xx} + U_y(I-B') + y (A+I) U_y + y U_y B + xU_x B\nonumber\\
& \quad  - x(A+I) U_x + AUB = 0,\label{3.24}
\end{align}
\end{theorem}
\begin{proof}
	Let 
\begin{align}
U &= G_1(A, B, B'; x, y) = \sum_{m,n =0}^{\infty} \mathcal{A}_{m, n} \, x^m \, y^n.
\end{align}
Then, we have
\begin{align}
U_{xx} &= \sum_{m, n = 0}^{\infty} m(m-1) \mathcal{A}_{m, n} \, x^{m-2} \, y^n, \quad U_{xy} = \sum_{m, n = 0}^{\infty} m \, n \mathcal{A}_{m, n} \, x^{m-1} \, y^{n-1}\nonumber\\
U_{yy} &= \sum_{m, n = 0}^{\infty} n(n-1) \mathcal{A}_{m, n} \, x^{m} \, y^{n-2}, \quad U_{x} = \sum_{m, n = 0}^{\infty} m \mathcal{A}_{m, n} \, x^{m-1} \, y^n\nonumber\\
U_{y} &= \sum_{m, n = 0}^{\infty} n \mathcal{A}_{m, n} \, x^{m} \, y^{n-1}.\label{3.26}
\end{align}
Using \eqref{3.26} in the left side of Equation \eqref{3.23}, we get
\begin{align}
&x (1 + x) U_{xx} - y U_{xy} - y^2 U_{yy} + U_x(I-B) + x (A+I) U_x + x U_x B' + yU_y B'\nonumber\\
& \quad  - y(A+I) U_y + AUB'\nonumber\\ 
& = \sum_{m,n=0}^{\infty} [ m(A + (m+n)I) \mathcal{A}_{m, n} (B' + (m-n)I) - n (A + (m+n)I) \mathcal{A}_{m, n}\nonumber\\
& \quad \times  (B' + (m-n)I) + m (m-1)  \mathcal{A}_{m, n} (B + (n-m-1)I) - n (n-1) \mathcal{A}_{m, n}\nonumber\\
& \quad \times  (B + (n-m-1)I) + (A + (m+n)I) \mathcal{A}_{m, n} (I-B) (B' + (m-n)I) \nonumber\\
& \quad + m (A+I) \mathcal{A}_{m, n} (B+(n-m-1)I) + m  \mathcal{A}_{m, n} B' (B + (n-m-1)I)\nonumber\\
& \quad - n (A+I)  \mathcal{A}_{m, n} (B + (n-m-1)I) + n  \mathcal{A}_{m, n} B' (B + (n-m-1)I)\nonumber\\
& \quad + A   \mathcal{A}_{m, n} B' (B + (n-m-1)I) ] (B+(n-m-1)I)^{-1} x^m y^n\label{3.27}
\end{align}
Using the commutativity of matrices $B$, $B'$ and the distributive property of matrices, the expression written in the big bracket in Equation \eqref{3.27} turns out to be $0$ (zero matrix). Hence, the matrix differential equation \eqref{3.23} is proved. Similarly, we are able to prove the Equation \eqref{3.24}.
\end{proof}
The systems of matrix differential equations of bilateral type satisfied by remaining nine Horn matrix functions and 13 confluent matrix functions are listed in the Table~\ref{t1}. The proofs are similar to theorem~\ref{t3.2} and hence are omitted.

\begin{longtable}{|l|l|c|}
	\hline
	Functions & Systems of Matrix Differential equations & Conditions \\
	\hline
	$G_2$ & $\begin{array}{c}
	x(1+x) U_{xx} - y(1+x)U_{xy}  + U_x(I-B) + x(A+I)U_x\\
	+ x U_x B' - yA U_y + A U B' = 0,\\
y(1+y) U_{yy} - x(1+y)U_{xy}  + U_y(I-B') + y(A'+I)U_y\\
+ y U_y B - xA' U_x + A' U B = 0;
	\end{array}$ &  $\begin{array}{c}
	AA' = A'A,\\
	BB' = B'B 
	\end{array}$\\
	\hline
	$G_3$ & $\begin{array}{c}
	x(1+4x) U_{xx} - (4x + 2) y U_{xy} + y^2 U_{yy} + (I-A)U_x \\
	+ x(4 A'+6I)U_x - 2yA'U_y +  A'(A'+I) U = 0,\\
y(1+4y) U_{yy} - (4y + 2) x U_{xy} + x^2 U_{xx} + (I-A')U_y \\
+ y(4 A+6I)U_y - 2xAU_x +  A(A+I) U = 0;  
	\end{array}$ &  ......\\
	\hline
	$H_1$ & $\begin{array}{c}
	x(1-x)U_{xx} + y^2 U_{yy} + U_x C' - x (A+I)U_x\\
	-   B xU_x  - yAU_y + (B+I) yU_y - A B U  = 0,\\
	-y(1+y)U_{yy} + x(1-y) U_{xy} + (A-I)U_y\\
- yU_y(C+I) - B yU_y -xU_x C - B U C = 0;
	\end{array}$ &  $\begin{array}{c}
	AB = BA,\\
	CC' = C'C 
	\end{array}$\\
	\hline
	$H_2$ & $\begin{array}{c}
	x(x-1)U_{xx} - xyU_{xy} + x (A+I)U_x + B xU_x \\
	- U_x C'' - B yU_y  + A B U = 0,\\
	y(1+y)U_{yy} - xU_{xy} + y (I-A)U_y\\
+ yU_y (C+C'+I) + UCC' = 0;
	\end{array}$ &  $\begin{array}{c}
AB = BA,\\
CC' = C'C,\\
	CC'' = C''C,\\
	C' C'' = C'' C' 
	\end{array}$\\
	\hline
	$H_3$ & $\begin{array}{c}
	x(1-4x)U_{xx} - y(1-4x)U_{xy} - y^2 U_{yy} + U_x C\\
	- x (4A+6I)U_x -2(A+I)yU_y - A(A+I) U = 0,\\
y(1-y)U_{yy} + x(1-2y)U_{xy}  + U_y C -(A+I)yU_y \\
-yU_y B - 2xU_x B  - A U B = 0;	
	\end{array}$ &  $\begin{array}{c}
	BC = CB
	\end{array}$\\
	\hline
	$H_{4}$ & $\begin{array}{c}
	x(1-4x) U_{xx} - 4xy U_{xy} - y^2 U_{yy} + U_x C - 4x(A + I)U_x\\
	- y (3A + 2I) U_y - A (A+I) U = 0,\\
	y(1-y)U_{yy} - 2xy U_{xy} +  U_{y} C'  - y A U_y \\
	 - B y U_y  - 2 B  xU_x - A B U = 0;
	\end{array}$ &  $\begin{array}{c}
	AB = BA,\\
	CC' = C'C   
	\end{array}$\\
	\hline
	$H_{5}$ & $\begin{array}{c}
	x(1+4x)U_{xx} -  y(1-4x)U_{xy} + y^2 U_{yy} +U_x (I-C)\\
	+ 4x(A+I) U_x + y(3A+2I) U_y -  A(A+I) U = 0,\\
	y(1-y)U_{yy} -  xyU_{xy} +  2 x^2U_{xx} + U_y C - (A + I)yU_y\\
	-  yU_y B + (A+2I) x U_x - 2x U_x B - A U B = 0;
	\end{array}$ &  $\begin{array}{c}
	BC = CB
	\end{array}$\\
	\hline
	$H_{6}$ & $\begin{array}{c}
	x(1+4x)U_{xx} -  y(1+4x)U_{xy} + y ^2 U_{yy} + U_{x} (I-B)\\
	+ (4A+6I) xU_x - 2A yU_y + A (A+I) U = 0,\\
	y(1+y)U_{yy} -  x(2+y)U_{xy}  + yU_{y} (B+C+I)\\+ (I-A) U_y - xU_x C + B U C = 0;
	\end{array}$ &  $\begin{array}{c}
BC = CB
	\end{array}$\\
	\hline
	$H_{7}$ & $\begin{array}{c}
	x(1-4x)U_{xx} + 4xy U_{xy} - y^2U_{yy} + U_x C \\
	 - x(4A+6I)U_x + 2AyU_y - A(A+I) U  = 0,\\
	y(1+y)U_{yy} -  3xyU_{xy}  + yU_y (C+I) + ByU_y\\ + (I-A)U_{y} - xU_x C + B U  C = 0;
	\end{array}$ &  $\begin{array}{c}
AB = BA,\\
CC' = C'C
	\end{array}$\\
	\hline
		$\Gamma_{1}$ & $\begin{array}{c}
		x(1+x)U_{xx}  - y(1+x) U_{xy} + (I-B)U_x + (A+I) x U_x  \\
		+ xU_x B' - A yU_y + A U B' = 0,\\
		y U_{yy} - A U_{xy}  + (1+y) U_y - U_y B'\\  - xU_x  + B U  = 0;
	\end{array}$ &  $\begin{array}{c}
		AB = BA
	\end{array}$\\
	\hline
		$\Gamma_{2}$ & $\begin{array}{c}
		x U_{xx}  - y U_{xy} + (I-B)U_x +  x U_x  + yU_y + UB'  = 0,\\
		y U_{yy} - x U_{xy}  + (1+y) U_y - U_y B' - xU_x  + B U  = 0;
	\end{array}$ &  $\begin{array}{c}
		...
	\end{array}$\\
	\hline
	$\mathcal{H}_{1}$ & $\begin{array}{c}
		x(1-x)U_{xx}  + y^2  U_{yy} + U_x C - (A+I) x U_x  \\
		+ (yU_y- xU_x) B + (I-A) yU_y - A U B = 0,\\
		y U_{yy} - x U_{xy}  + (I-A) U_y + y U_y \\  + xU_x  +  U B = 0;
	\end{array}$ &  $\begin{array}{c}
		BC = CB
	\end{array}$\\
\hline
	$\mathcal{H}_{2}$ & $\begin{array}{c}
	x(1-x)U_{xx}  + xy  U_{xy} + U_x C - (A+I) x U_x  \\
	+ B (yU_y- xU_x)   - A BU  = 0,\\
	y U_{yy} - x U_{xy}  + (I-A) U_y + y U_y  +  U B' = 0;
\end{array}$ &  $\begin{array}{c}
AB = BA,\\
B'C = CB'
\end{array}$\\
\hline
	$\mathcal{H}_{3}$ & $\begin{array}{c}
	x(1-x)U_{xx}  + xy  U_{xy} + U_x C - (A+I) x U_x  \\
	+ (yU_y- xU_x) B  - A U B = 0,\\
	y U_{yy} - x U_{xy}  + (I-A) U_y + y U_y   +  U  = 0;
\end{array}$ &  $\begin{array}{c}
	BC = CB
\end{array}$\\
\hline
	$\mathcal{H}_{4}$ & $\begin{array}{c}
	x U_{xx}  + U_x C -  x U_x 	+ yU_y - A U  = 0,\\
	y U_{yy} - x U_{xy}  + (I-A) U_y + y U_y   +  U B' = 0;
\end{array}$ &  $\begin{array}{c}
	B'C = CB'
\end{array}$\\
\hline
	$\mathcal{H}_{5}$ & $\begin{array}{c}
	xU_{xx}  +  U_x C + yU_y- xU_x  - A U = 0,\\
	y U_{yy} - x U_{xy}  + (I-A) U_y + y U_y  +  U = 0;
\end{array}$ &  $\begin{array}{c}
...
\end{array}$\\
\hline	$\mathcal{H}_{6}$ & $\begin{array}{c}
	x(1-4x)U_{xx} + y(1-4x)U_{xy} - y^2  U_{yy} + U_x C \\ - (4A+6I) x U_x - (2A+2I) yU_y  - A (A+I)U = 0,\\
	y U_{yy} + x U_{xy}  +  U_y C - y U_y   -2 xU_x  -A  U = 0;
\end{array}$ &  $\begin{array}{c}
	...
\end{array}$\\
\hline
	$\mathcal{H}_{7}$ & $\begin{array}{c}
		x(1-4x)U_{xx}  - 4xy U_{xy} - y^2  U_{yy} + U_x C \\ - (4A+4I) x U_x - (3A+2I) yU_y  - A (A+I)U = 0,\\
		y U_{yy}  +  U_y C' - y U_y   -2 xU_x  -A  U = 0;
	\end{array}$ &  $\begin{array}{c}
	CC' = C'C
	\end{array}$\\
\hline
	$\mathcal{H}_{8}$ & $\begin{array}{c}
	x(1+4x)U_{xx} + y(1+4x)U_{xy} + y^2  U_{yy} + U_x (I-B) \\ + (4A+6I) x U_x  + 2A yU_y  + A (A+I)U = 0,\\
	y U_{yy} -2 x U_{xy}  + (I-A) U_y + y U_y   - xU_x  +  UB = 0;
\end{array}$ &  $\begin{array}{c}
	...
\end{array}$\\
\hline
$\mathcal{H}_{9}$ & $\begin{array}{c}
	x(1-4x)U_{xx}  + 4xy U_{xy} - y^2  U_{yy} + U_x C \\ - (4A+6I) x U_x + 2A yU_y  - A (A+I)U = 0,\\
	y U_{yy} - 2xU_{xy} +  (I-A)U_y  + y U_y   + UB = 0;
\end{array}$ &  $\begin{array}{c}
	BC = CB
\end{array}$\\
\hline
$\mathcal{H}_{10}$ & $\begin{array}{c}
	x(1-4x)U_{xx}  + 4xy U_{xy} - y^2  U_{yy} + U_x C \\ - (4A+6I) x U_x + 2A yU_y  - A (A+I)U = 0,\\
	y U_{yy} - 2xU_{xy} +  (I-A)U_y    + U = 0;
\end{array}$ &  $\begin{array}{c}
...
\end{array}$\\
\hline
$\mathcal{H}_{11}$ & $\begin{array}{c}
	x U_{xx}   + U_x C' -  x U_x +  yU_y  - A U = 0,\\
	y (1+y)U_{yy} - xU_{xy} +  (I-A)U_y  + (I+B)y U_y \\  +y U_y C  + BUC = 0;
\end{array}$ &  $\begin{array}{c}
AB = BA,\\
CC' = C'C.
\end{array}$\\
\hline
	\caption{Systems of partial matrix differential equations of bilateral type satisfied by Horn matrix functions and confluent matrix functions}\label{t1}
\end{longtable}

\subsection{Certain Integral Representations}
We now give the integral representation of some Horn matrix functions. Starting with the integral representation of $G_1(A, B, B'; x, y)$, presented in the following theorem:
\begin{theorem}
	For positive stable matrices $A$, $B$, $B' \in \mathbb{C}^{r \times r}$ such that $B B' = B' B$. The Horn matrix function $G_1(A, B, B'; x, y)$ can be presented in the integral form as:
	\begin{align}
		G_1(A, B, B'; x, y) = \int_{0}^{1} \left(1 + \frac{x}{t} + y t \right)^{-A} t^{B- I} (1 - t)^{-(B + B')} dt \times \Gamma\left(\begin{array}{c}
		I - B'\\
		B, I - B - B'
		\end{array}\right). \label{3.1.0} 
	\end{align}
\end{theorem}
\begin{proof}
	Using the matrix identity $(A)_{-n} = (-1)^n (I - A)_n^{-1}$ in \eqref{3.1}, we get
	\begin{align}
	G_1(A, B, B'; x, y) & = \sum_{m, n \ge 0} (A)_{m + n} (B)_{n - m} (-1)^{m + n} (I - B')_{n - m}^{-1} \frac{x^m \, y^n}{m ! \, n!}\nonumber\\
	& = \sum_{m, n \ge 0} (A)_{m + n} \frac{(-x)^m \, (-y)^n}{m ! \, n!} (B)_{n - m} \, (I - B')_{n - m}^{-1}. \label{3.1.1}
	\end{align}
Now, using the integral representation of Pochammer symbol
\begin{align}
	(A)_m \, (C)_m^{-1} = \Gamma(C) \, \Gamma^{-1}(A) \, \Gamma^{-1} (C-A) \int_{0}^{1} t^{A + (m-1)I} (1 - t)^{C-A-I} dt, \quad AC = CA
\end{align} 
in \eqref{3.1.1}, we get
\begin{align}
	G_1(A, B, B'; x, y) & =  \sum_{m, n \ge 0} (A)_{m + n} \frac{(-x)^m \, (-y)^n}{m ! \, n!} \,  \int_{0}^{1}  t^{B + (n - m - 1)I} (1 - t)^{-(B + B')} dt\nonumber\\
	& \quad  \times \Gamma\left(\begin{array}{c}
		I - B'\\
		B, I - B - B'
	\end{array}\right). \label{3.1.2}
\end{align}
The matrix identity $(1 - x - y)^{-A} = \sum_{m, n = 0}^{\infty} (A)_{m + n} \frac{x^m \, y^n}{m ! \, n!}$ and the equation \eqref{3.1.2} together yield the integral representation \eqref{3.1.0}.
\end{proof}
Next, we give the integral representations of $G_2$, $H_3$ and $H_4$ presented in the theorems below. Since the proofs are similar to $G_1$, so we omit them.
\begin{theorem}
Let $A$, $A'$, $B$, $B'$, $I - B'$, $I - B - B'$  be  positive stable matrices in $\mathbb{C}^{r \times r}$ such that $B B' = B' B$. Then, the Horn matrix function $G_2 (A, A', B, B'; x, y)$ can be presented in the integral form as:
\begin{align}
G_2 (A, A', B, B'; x, y) & = \int_{0}^{1} \left(1 + \frac{x}{t}\right)^{-A} \, (1 + y \, t)^{-A'} \, t^{B - I} \, (1 - t)^{-(B + B')} \, dt \nonumber\\
&  \quad  \times   \Gamma (I - B') \, \Gamma^{-1} (B) \, \Gamma^{-1} (I - B - B').
\end{align}
\end{theorem}
\begin{theorem}
	Let $A$, $B$, $C$, $C - A$ be positive stable matrices in $\mathbb{C}^{r \times r}$ such that $AB = BA$, $AC = CA$. Then the matrix function $H_3 (A, B; C; x, y)$ can be put in the integral form as
	\begin{align}
	H_3 (A, B; C; x, y) & = \int_{0}^{1} (1 - y \, t)^{-B} \, t^{A - I} \left(1 + \frac{x \, t^2}{1 - t}\right)^{-C - A - I} \,  (1 - t)^{C - A - I} \, dt \nonumber\\
	&  \quad  \times   \Gamma (C) \, \Gamma ^{-1} (A) \, \Gamma^{-1} (C - A). 
	\end{align}
\end{theorem}
\begin{theorem}
	Let $A$, $B$, $C$, $C'$, $C - A$, $C' - B$ be commuting and  positive stable matrices in $\mathbb{C}^{r \times r}$. Then the matrix function $H_4 (A, B; C, C'; x, y)$ can be put in the integral form as
	\begin{align}
	&	H_4 (A, B; C, C'; x, y)\nonumber\\
	& = \int_{0}^{1}\int_{0}^{1} t^{A - I} \, (1 -  t)^{C - A - I} \, u^{B - I} \left(1 - u\right)^{C' - B - I} \,  (1 - u y)^{ - A} \, \left(1 - \frac{tx}{(1 - u y)^2}\right) \, dt \, du \nonumber\\
		&  \quad  \times   \Gamma (C) \, \Gamma (C') \, \Gamma ^{-1} (A) \, \Gamma ^{-1} (B) \, \Gamma^{-1} (C - A) \, \Gamma^{-1} (C' - B). 
	\end{align}
\end{theorem}
We have not given here the integral formula for remaining Horn's functions since they do not culminate into appropriate integral.

\section{Differential formulae}
In this section, we give the differential formulae satisfied by  Horn matrix functions and their confluent cases. The differential formulae for first Horn matrix function $G_1(A, B, B'; x, y)$ are given in the theorem below:
\begin{theorem}\label{t5.1}
	Let $A$, $B$ and $B'$ be matrices in $\mathbb{C}^{r \times r}$. Then the Horn matrix function $G_1(A, B, B'; x, y)$ satisfies the following differential formulae
\begin{align}
&	\frac{\partial ^r}{\partial x^r} G_1(A, B, B'; x, y)\nonumber\\
& = (-1)^r \, (A)_r \, G_1(A+rI, B-rI, B'+rI; x, y) \, (I-B)_r^{-1} \, (B')_r, \quad BB' = B'B;\label{e5.1}
\\[5pt]
&	\frac{\partial ^r}{\partial y^r} G_1(A, B, B'; x, y)\nonumber\\
& = (-1)^r \, (A)_r \, G_1(A+rI, B+rI, B' - rI; x, y) \, (B)_r \, (I-B')_r^{-1}, \quad BB' = B'B;\label{e5.2}
\\[5pt]
& \left(x^2 \frac{\partial }{\partial x}\right)^r \ [x^{A+(r-1)I} G_1(A, B, B'; x, xy)]  = x^{A+rI} \, (A)_r \, G_1(A+rI, B, B'; x, xy);\label{e5.3}
\\[5pt]
& \left(y^2 \frac{\partial }{\partial y}\right)^r \ [y^{A+(r-1)I} G_1(A, B, B'; xy, y)]  = y^{A+rI} \, (A)_r \, G_1(A+rI, B, B'; xy, y);\label{e5.4}
\\[5pt]
& \left(x^2 \frac{\partial }{\partial x}\right)^r \ [ G_1(A, B, B'; x, \frac{y}{x}) \, x^{B'+(r-1)I}]  =  G_1(A, B, B'+rI; x, \frac{y}{x}) \ x^{B'+rI} \, (B')_r;\label{e5.5}
\\[5pt]
& \left(y^2 \frac{\partial }{\partial y}\right)^r \ [G_1(A, B, B'; \frac{x}{y}, y) \, y^{B + (r-1)I}]  =  G_1(A, B+rI, B'; \frac{x}{y}, y) \, y^{B+rI} \, (B)_r.\label{e5.6}
\end{align}
\end{theorem}
\begin{proof}
From equation \eqref{3.1}, we have 
\begin{align}
		\frac{\partial}{\partial x} G_1(A, B, B'; x, y) &  = \sum_{m,n =0}^{\infty} (A)_{m+n} (B)_{n-m} (B')_{m-n} \ \frac{\partial}{\partial x} \frac{x^m \, y^n}{m!\, n!}\nonumber\\
	& = \sum_{m = 1,n =0}^{\infty} (A)_{m+n} (B)_{n-m} (B')_{m-n} \   \frac{x^{m-1} \, y^n}{(m-1)!\, n!}\nonumber\\
	& = \sum_{m,n =0}^{\infty} (A)_{m+n+1} (B)_{n-m-1} (B')_{m-n+1} \   \frac{x^{m} \, y^n}{m! \, n!}\nonumber\\
	& = (-1) (A)_1 \sum_{m,n =0}^{\infty} (A+I)_{m+n} (B-I)_{n-m} (B'+I)_{m-n} \   \frac{x^{m} \, y^n}{m! \, n!}\nonumber\\
	& \quad \times  (I-B)_1^{-1} (B')_1\nonumber\\
	& = (-1) (A)_1 \, G_1(A+I, B-I, B'+I; x, y) (I-B)_1^{-1} (B')_1.
\end{align}
Iterating this process $r$-times, we get required formula \eqref{e5.1}. In the similar way, we can proof \eqref{e5.2}. Now, to prove \eqref{e5.3}, consider the left hand side 
\begin{align}
&\left(x^2 \frac{\partial }{\partial x}\right) \ [x^{A} G_1(A, B, B'; x, xy)] \nonumber\\
& = \sum_{m, n = 0}^{\infty} \left(x^2 \frac{\partial }{\partial x}\right) x^{A+(m+n)I} (A)_{m+n} (B)_{n-m} (B')_{m-n} \  \frac{y^n}{m! \, n!}\nonumber\\
& =  \sum_{m, n = 0}^{\infty} (A+(m+n)I) x^{A+(m+n+1)I} (A)_{m+n} (B)_{n-m} (B')_{m-n} \  \frac{y^n}{m! \, n!}.\label{e5.8}
\end{align}  
Using the identity $(A+(m+n)I) \, (A)_{m+n} = A \, (A+I)_{m+n}$, equation \eqref{e5.8} yields
\begin{align}
&	\left(x^2 \frac{\partial }{\partial x}\right) \ [x^{A} G_1(A, B, B'; x, xy)]\nonumber\\
 & =   \sum_{m, n = 0}^{\infty} A \, x^{A + I} (A+I)_{m+n}  (B)_{n-m} (B')_{m-n} \  \frac{x^m (xy)^n}{m! \, n!}\nonumber\\
	& = (A)_1 \, x^{A + I} \,  \sum_{m, n = 0}^{\infty} (A+I)_{m+n}  (B)_{n-m} (B')_{m-n} \  \frac{x^m (xy)^n}{m! \, n!}\nonumber\\
	& = (A)_1 \, x^{A + I} \, G_1(A+I, B, B'; x, xy).
\end{align} 
Similarly the other formulae from \eqref{e5.4} to \eqref{e5.6} can be proved.
\end{proof}
Next, we give the differential formulas satisfy by the remaining nine Horn matrix functions and confluent matrix functions. Since the proofs are similar to theorem~\ref{t5.1} so we omit them. 
\begin{theorem}
	Let $A$, $A'$, $B$ and $B'$ be matrices in $\mathbb{C}^{r \times r}$. Then the Horn matrix function $G_2(A, A', B, B'; x, y)$ satisfies the following differential formulae
	\begin{align}
		&	\frac{\partial ^r}{\partial x^r} G_2(A, A', B, B'; x, y)\nonumber\\
		& = (-1)^r \, (A)_r \, G_2(A+rI, A', B-rI, B'+rI; x, y) \, (I-B)_r^{-1} \, (B')_r,\nonumber\\
		& \quad BB' = B'B;
		\\[5pt]
		&	\frac{\partial ^r}{\partial y^r} G_2(A, A', B, B'; x, y)\nonumber\\
		& = (-1)^r \, (A')_r \, G_2(A, A'+rI, B+rI, B' - rI; x, y) \, (B)_r \, (I-B')_r^{-1},\nonumber\\
		& \quad AA' = A'A,  BB' = B'B;
		\\[5pt]
		& \left(x^2 \frac{\partial }{\partial x}\right)^r \ [x^{A+(r-1)I} G_2(A, A', B, B'; x, y)]\nonumber\\
		&  = x^{A+rI} \, (A)_r \, G_2(A+rI, A', B, B'; x, y);
		\\[5pt]
		& \left(y^2 \frac{\partial }{\partial y}\right)^r \ [y^{A'+(r-1)I} G_2(A, A', B, B'; x, y)] \nonumber\\
		& = y^{A'+rI} \, (A')_r \, G_1(A, A'+rI, B, B'; x, y), \quad AA' = A'A;
		\\[5pt]
		& \left(x^2 \frac{\partial }{\partial x}\right)^r \ [ G_2(A, A', B, B'; x, \frac{y}{x}) \, x^{B'+(r-1)I}]\nonumber\\
		&  =  G_2(A, A', B, B'+rI; x, \frac{y}{x}) \ x^{B'+rI} \, (B')_r;
		\\[5pt]
		& \left(y^2 \frac{\partial }{\partial y}\right)^r \ [G_2(A, A', B, B'; \frac{x}{y}, y) \, y^{B + (r-1)I}]\nonumber\\
		&  =  G_2(A, A', B+rI, B'; \frac{x}{y}, y) \, y^{B+rI} \, (B)_r, \quad BB' = B'B.
	\end{align}
\end{theorem}
\begin{theorem}
	Let $A$ and $A'$ be matrices in $\mathbb{C}^{r \times r}$. Then the Horn matrix function $G_3(A, A'; x, y)$ satisfies the following differential formulae
	\begin{align}
		&	\frac{\partial ^r}{\partial x^r} G_3(A, A'; x, y)  = (-1)^r \, (I-A)^{-1}_r \, G_3(A-rI, A'+2rI; x, y) \,  (A')_{2r};
		\\[5pt]
		&	\frac{\partial ^r}{\partial y^r} G_3(A, A'; x, y)  = (-1)^r \, (A)_{2r} \, G_3(A+2rI, A'-rI; x, y) \, (I-A')_r^{-1};
		\\[5pt]
		& \left(x^2 \frac{\partial }{\partial x}\right)^r \ [G_3(A, A'; x^2, \frac{y}{x}) x^{A'+(r-1)I} ] =  G_3 (A, A'+rI; x^2, \frac{y}{x}) x^{A'+rI} \, (A')_r;
		\\[5pt]
		& \left(y^2 \frac{\partial }{\partial y}\right)^r \ [y^{A+(r-1)I} G_3(A, A'; \frac{x}{y}, y^2)]   = y^{A+rI} \, (A)_r \, G_3(A+rI, A'; \frac{x}{y}, y^2).
	\end{align}
\end{theorem}
\begin{theorem}
	Let $A$, $B$, $C$ and $C'$ be matrices in $\mathbb{C}^{r \times r}$. Then the Horn matrix function $H_1(A, B, C, C'; x, y)$ satisfies the following differential formulae
	\begin{align}
		&	\frac{\partial ^r}{\partial x^r} H_1(A, B, C, C'; x, y)\nonumber\\
		& =  (A)_r \, (B)_r \, H_1(A+rI, B+rI, C, C'+rI; x, y) \,   (C')^{-1}_r,\quad AB = BA;
		\\[5pt]
		&	\frac{\partial ^r}{\partial y^r} H_1(A, B, C, C'; x, y)\nonumber\\
		& = (-1)^r \, (I-A)^{-1}_r \, (B)_r \nonumber\\
		& \quad \times  H_1(A-rI, B+rI, C+rI, C'; x, y) \, (C)_r, \quad AB = BA,  CC' = C'C;
		\\[5pt]
		& \left(x^2 \frac{\partial }{\partial x}\right)^r \ [x^{A+(r-1)I} H_1(A, B, C, C'; x, \frac{y}{x})]  = x^{A+rI} \, (A)_r \, H_1(A+rI, B, C, C'; x, \frac{y}{x});
		\\[5pt]
		& \left(x^2 \frac{\partial }{\partial x}\right)^r \ [x^{B+(r-1)I} H_1(A, B, C, C'; x, {y}{x})]\nonumber\\
		&  = x^{B+rI} \, (B)_r \, H_1(A, B+rI, C, C'; x, {y}{x}), \quad AB = BA;
		\\[5pt]
		&  \left(y^2 \frac{\partial }{\partial y}\right)^r \ [ H_1(A, B, C, C'; x, {y}) y^{C+(r-1)I}]\nonumber\\
		&  =  H_1(A, B, C + rI, C'; x, {y}) \,  x^{C+rI} \, (C)_r, \quad CC'= C'C;
		\\[5pt]
		&   \frac{\partial^r }{\partial x^r} \ [ H_1(A, B, C, C'; x, {y}) x^{C'-I}]\nonumber\\
		&  =  (-1)^r \, H_1(A, B, C, C'-rI; x, {y}) \, (I-C')_r \, x^{C'-(r+1)I}.
	\end{align}
\end{theorem}
\begin{theorem}
	Let $A$, $B$, $C$, $C'$ and $C''$ be matrices in $\mathbb{C}^{r \times r}$. Then the Horn matrix function $H_2(A, B, C, C', C''; x, y)$ satisfies the following differential formulae
	\begin{align}
		&	\frac{\partial ^r}{\partial x^r} H_2(A, B, C, C', C''; x, y)\nonumber\\
		& =  (A)_r \, (B)_r \, H_2(A+rI, B+rI, C, C', C''+rI; x, y) \,   (C'')^{-1}_r,\quad AB = BA;
		\\[5pt]
		&	\frac{\partial ^r}{\partial y^r} H_2(A, B, C, C', C''; x, y)\nonumber\\
		& = (-1)^r \, (I-A)^{-1}_r   H_2(A-rI, B, C+rI, C'+rI, C''; x, y) \, (C)_r \, (C')_r, \nonumber\\
		& \qquad  CC' = C'C, CC'' = C''C, C'C'' = C''C';
		\\[5pt]
		& \left(x^2 \frac{\partial }{\partial x}\right)^r \ [x^{A+(r-1)I} H_2(A, B, C, C', C''; x, \frac{y}{x})]\nonumber\\
		&  = x^{A+rI} \, (A)_r \, H_2(A+rI, B, C, C', C''; x, \frac{y}{x});
		\\[5pt]
		& \left(x^2 \frac{\partial }{\partial x}\right)^r \ [x^{B+(r-1)I} H_2(A, B, C, C', C''; x, {y})]\nonumber\\
		&  = x^{B+rI} \, (B)_r \, H_2(A, B+rI, C, C', C''; x, {y}), \quad AB = BA;
		\\[5pt]
		&  \left(y^2 \frac{\partial }{\partial y}\right)^r \ [ H_2(A, B, C, C', C''; x, {y}) y^{C+(r-1)I}]\nonumber\\
		&  =  H_2(A, B, C + rI, C', C''; x, {y}) \,  y^{C+rI} \, (C)_r, \quad CC'= C'C, CC'' = C''C;
		\\[5pt]
		&   \frac{\partial^r }{\partial x^r} \ [ H_2(A, B, C, C', C''; x, {y}) x^{C''-I}]\nonumber\\
		&  =  (-1)^r \, H_2(A, B, C, C', C''-rI; x, {y}) \, (I-C'')_r \, x^{C''-(r+1)I}.
	\end{align}
\end{theorem}
\begin{theorem}
	Let $A$, $B$, $C$ be matrices in $\mathbb{C}^{r \times r}$. Then the Horn matrix function $H_3(A, B; C; x, y)$ satisfies the following differential formulae
	\begin{align}
		&	\frac{\partial ^r}{\partial x^r} H_3(A, B; C; x, y) =  (A)_{2r}  \, H_3(A+2rI, B; C+rI; x, y) \,   (C)^{-1}_r;
		\\[5pt]
		&	\frac{\partial ^r}{\partial y^r} H_3(A, B; C; x, y)\nonumber\\
		& = (A)_r  \,  H_3(A+rI, B+rI; C+rI; x, y) \, (B)_r \, (C)^{-1}_r, \quad BC = CB;
		\\[5pt]
		& \left(x^2 \frac{\partial }{\partial x}\right)^r \ [x^{A+(r-1)I} H_3(A, B; C; x^2, {y}{x})]\nonumber\\
		&  = x^{A+rI} \, (A)_r \, H_3(A+rI, B; C; x^2, {y}{x});
		\\[5pt]
		&  \left(y^2 \frac{\partial }{\partial y}\right)^r \ [ H_3(A, B; C; x, {y}) y^{B+(r-1)I}]\nonumber\\
		&  =  H_3(A, B+rI; C; x, {y}) \,  y^{B+rI} \, (B)_r, \quad BC = CB;
		\\[5pt]
		&   \frac{\partial^r }{\partial x^r} \ [ H_3(A, B; C; x, {x y}) x^{C-I}]  =  (-1)^r \, H_3(A, B; C-rI; x, {y}) \, x^{C-(r+1)I} \, (I-C)_r.
	\end{align}
\end{theorem}
\begin{theorem}
	Let $A$, $B$, $C$ and $C'$ be matrices in $\mathbb{C}^{r \times r}$. Then the Horn matrix function $H_4(A, B; C, C'; x, y)$ satisfies the following differential formulae
	\begin{align}
		&	\frac{\partial ^r}{\partial x^r} H_4(A, B; C, C'; x, y)\nonumber\\
		& =  (A)_{2r}  \, H_4(A+2rI, B; C+rI, C'; x, y) \,   (C)^{-1}_r;
		\\[5pt]
		&	\frac{\partial ^r}{\partial y^r} H_4(A, B; C, C'; x, y)\nonumber\\
		& = (A)_r \, (B)_r \,  H_4(A+rI, B+rI; C, C'+rI; x, y) \,  (C')_r, \quad AB = BA;
		\\[5pt]
		& \left(x^2 \frac{\partial }{\partial x}\right)^r \ [x^{A+(r-1)I} H_4(A, B; C, C'; x^2, {y}{x})]\nonumber\\
		&  = x^{A+rI} \, (A)_r \, H_4(A+rI, B; C, C'; x^2, {y}{x});
		\\[5pt]
		&  \left(y^2 \frac{\partial }{\partial y}\right)^r \ [  y^{B+(r-1)I} \, H_4(A, B; C, C'; x, {y})]\nonumber\\
		&  = y^{B+rI} (B)_r \, H_4(A, B+rI; C, C'; x, {y}), \quad AB = BA;
		\\[5pt]
		&   \frac{\partial^r }{\partial x^r} \ [ H_4(A, B; C, C'; x, {y}) x^{C-I}]\nonumber\\
		&  =  (-1)^r \, H_4(A, B; C-rI, C'; x, {y}) \, x^{C-(r+1)I} \,  (I-C)_r, \quad CC' = C'C;
		\\[5pt]
		&   \frac{\partial^r }{\partial y^r} \ [ H_4(A, B; C, C'; x, {y}) y^{C'-I}]\nonumber\\
		&  =  (-1)^r \, H_4(A, B; C, C'-rI; x, {y}) \, y^{C'-(r+1)I} \,  (I-C')_r.
	\end{align}
\end{theorem}
\begin{theorem}
	Let $A$, $B$ and $C$ be matrices in $\mathbb{C}^{r \times r}$. Then the Horn matrix function $H_5(A, B; C; x, y)$ satisfies the following differential formulae
	\begin{align}
		&	\frac{\partial ^r}{\partial x^r} H_5(A, B; C; x, y)\nonumber\\
		& =  (-1)^r (A)_{2r} \, (I-B)_r^{-1}\, H_5(A+2rI, B-rI; C; x, y), \quad AB = BA;
		\\[5pt]
		&	\frac{\partial ^r}{\partial y^r} H_5(A, B; C; x, y)\nonumber\\
		& = (A)_r \, (B)_r \,  H_5(A+rI, B+rI; C+rI; x, y) \,  (C)^{-1}_r, \quad AB = BA;
		\\[5pt]
		& \left(x^2 \frac{\partial }{\partial x}\right)^r \ [x^{A+(r-1)I} H_5(A, B; C; x^2, {y}{x})]\nonumber\\
		&  = x^{A+rI} \, (A)_r \, H_5(A+rI, B; C; x^2, {y}{x});
		\\[5pt]
		&  \left(y^2 \frac{\partial }{\partial y}\right)^r \ [  y^{B+(r-1)I} \, H_5(A, B; C; \frac{x}{y}, {y})]\nonumber\\
		&  = y^{B+rI} (B)_r \, H_5(A, B+rI; C; \frac{x}{y}, {y}), \quad AB = BA;
		\\[5pt]
		&   \frac{\partial^r }{\partial y^r} \ [ H_5(A, B; C; x, {y}) y^{C-I}]\nonumber\\
		&  =  (-1)^r \, H_5(A, B; C-rI; x, {y}) \, y^{C-(r+1)I} \,  (I-C)_r.
	\end{align}
\end{theorem}
\begin{theorem}
	Let $A$, $B$ and $C$ be matrices in $\mathbb{C}^{r \times r}$. Then the Horn matrix function $H_6(A, B; C; x, y)$ satisfies the following differential formulae
	\begin{align}
		&	\frac{\partial ^r}{\partial x^r} H_6(A, B; C; x, y)\nonumber\\
		& =  (-1)^r (A)_{2r} \, (I-B)_r^{-1}\, H_6(A+2rI, B-rI; C; x, y), \quad AB = BA;
		\\[5pt]
		&	\frac{\partial ^r}{\partial y^r} H_6(A, B; C; x, y)\nonumber\\
		& = (-1)^r (I-A)^{-1}_r \, (B)_r \,  H_6(A-rI, B+rI; C+rI; x, y) \,  (C)_r, \quad AB = BA;
		\\[5pt]
		& \left(x^2 \frac{\partial }{\partial x}\right)^r \ [x^{A+(r-1)I} H_6(A, B; C; x^2, \frac{y}{x})]\nonumber\\
		&  = x^{A+rI} \, (A)_r \, H_6(A+rI, B; C; x^2, \frac{y}{x});
		\\[5pt]
		&  \left(y^2 \frac{\partial }{\partial y}\right)^r \ [  y^{B+(r-1)I} \, H_6(A, B; C; \frac{x}{y}, {y})]\nonumber\\
		&  = y^{B+rI} (B)_r \, H_6(A, B+rI; C; \frac{x}{y}, {y}), \quad AB = BA;
		\\[5pt]
		&   \left(y^2 \frac{\partial }{\partial y}\right)^r \ [ H_6(A, B; C; x, {y}) y^{C+(r-1)I}]\nonumber\\
		&  =   \, H_6(A, B; C+rI; x, {y}) \, y^{C+rI} \,  (C)_r.
	\end{align}
\end{theorem}
\begin{theorem}
	Let $A$, $B$, $C$ and $C'$ be matrices in $\mathbb{C}^{r \times r}$. Then the Horn matrix function $H_7(A, B, C; C'; x, y)$ satisfies the following differential formulae
	\begin{align}
		&	\frac{\partial ^r}{\partial x^r} H_7(A, B, C; C'; x, y)\nonumber\\
		& =  (A)_{2r}  \, H_7(A+2rI, B, C; C'+rI; x, y) \,   (C')^{-1}_r;
		\\[5pt]
		&	\frac{\partial ^r}{\partial y^r} H_7(A, B, C; C'; x, y)\nonumber\\
		& = (-1)^r \, (I-A)^{-1}_r  (B)_r\nonumber\\
		& \quad \times   H_7(A-rI, B+rI, C+rI; C'; x, y) \, (C)_r, \quad AB + BA,  CC' = C'C;
		\\[5pt]
		& \left(x^2 \frac{\partial }{\partial x}\right)^r \ [x^{A+(r-1)I} H_7(A, B, C; C'; x^2, \frac{y}{x})]\nonumber\\
		&  = x^{A+rI} \, (A)_{r} \, H_7(A+rI, B, C; C'; x^2, \frac{y}{x});
		\\[5pt]
		& \left(y^2 \frac{\partial }{\partial y}\right)^r \ [x^{B+(r-1)I} H_7(A, B, C; C'; x, {y})]\nonumber\\
		&  = x^{B+rI} \, (B)_r \, H_7(A, B+rI, C; C'; x, {y}), \quad AB = BA;
		\\[5pt]
		&  \left(y^2 \frac{\partial }{\partial y}\right)^r \ [ H_7(A, B, C; C'; x, {y}) y^{C+(r-1)I}]\nonumber\\
		&  =  H_7(A, B, C + rI; C'; x, {y}) \,  y^{C+rI} \, (C)_r, \quad CC'= C'C;
		\\[5pt]
		&   \frac{\partial^r }{\partial x^r} \ [ H_7(A, B, C; C'; x, {y}) x^{C'-I}]\nonumber\\
		&  =  (-1)^r \, H_7(A, B, C; C'-rI; x, {y}) \, x^{C'-(r+1)I} \, (I-C')_r.
	\end{align}
\end{theorem} 
\begin{theorem}\label{t6.1}
	Let $A$, $B$ and $B'$ be matrices in $\mathbb{C}^{r \times r}$. Then the matrix function $\Gamma_1(A, B, B'; x, y)$ satisfies the following differential formulae
	\begin{align}
		&	\frac{\partial ^r}{\partial x^r} \Gamma_1(A, B, B'; x, y)\nonumber\\
		& = (-1)^r \, (A)_r \, \Gamma_1(A+rI, B-rI, B'+rI; x, y) \, (I-B)_r^{-1} \, (B')_r, \quad BB' = B'B;
		\\[5pt]
		&	\frac{\partial ^r}{\partial y^r} \Gamma_1(A, B, B'; x, y)\nonumber\\
		& = (-1)^r  \, \Gamma_1(A, B+rI, B' - rI; x, y) \, (B)_r \, (I-B')_r^{-1}, \quad BB' = B'B;
		\\[5pt]
		& \left(x^2 \frac{\partial }{\partial x}\right)^r \ [x^{A+(r-1)I} \Gamma_1(A, B, B'; x, y)]  = x^{A+rI} \, (A)_r \, \Gamma_1(A+rI, B, B'; x, y);
		\\[5pt]
		& \left(x^2 \frac{\partial }{\partial x}\right)^r \ [ \Gamma_1(A, B, B'; x, \frac{y}{x}) \, x^{B'+(r-1)I}]  =  \Gamma_1(A, B, B'+rI; x, \frac{y}{x}) \ x^{B'+rI} \, (B')_r;
		\\[5pt]
		& \left(y^2 \frac{\partial }{\partial y}\right)^r \ [\Gamma_1(A, B, B'; \frac{x}{y}, y) \, y^{B + (r-1)I}]  =  \Gamma_1(A, B+rI, B'; \frac{x}{y}, y) \, y^{B+rI} \, (B)_r.
	\end{align}
\end{theorem} 
\begin{theorem}
	Let $B$ and $B'$ be matrices in $\mathbb{C}^{r \times r}$. Then the matrix function $\Gamma_2( B, B'; x, y)$ satisfies the following differential formulae
	\begin{align}
		&	\frac{\partial ^r}{\partial x^r} \Gamma_2( B, B'; x, y) = (-1)^r \, (I-B)_r^{-1} \, \Gamma_2( B-rI, B'+rI; x, y) \,  (B')_r;
		\\[5pt]
		&	\frac{\partial ^r}{\partial y^r} \Gamma_2( B, B'; x, y) = (-1)^r \, (B)_r \, \Gamma_2( B+rI, B' - rI; x, y) \,  (I-B')_r^{-1};
		\\[5pt]
		& \left(x^2 \frac{\partial }{\partial x}\right)^r \ [ \Gamma_2(B, B'; x, \frac{y}{x}) \, x^{B'+(r-1)I}]  =  \Gamma_2( B, B'+rI; x, \frac{y}{x}) \ x^{B'+rI} \, (B')_r;
		\\[5pt]
		& \left(y^2 \frac{\partial }{\partial y}\right)^r \ [\Gamma_2(B, B'; \frac{x}{y}, y) \, y^{B + (r-1)I}]  =  \Gamma_2(B+rI, B'; \frac{x}{y}, y) \, y^{B+rI} \, (B)_r.
	\end{align}
\end{theorem}
\begin{theorem}
	Let $A$, $B$ and $C$ be matrices in $\mathbb{C}^{r \times r}$. Then the matrix function $\mathcal{H}_1(A, B; C; x, y)$ satisfies the following differential formulae
	\begin{align}
		&	\frac{\partial ^r}{\partial x^r} \mathcal{H}_1(A, B; C; x, y)\nonumber\\
		& =  (A)_r \, (B)_r \, \mathcal{H}_1(A+rI, B+rI; C+rI; x, y) \,   (C)^{-1}_r,\quad AB = BA;
		\\[5pt]
		&	\frac{\partial ^r}{\partial y^r} \mathcal{H}_1(A, B; C; x, y)\nonumber\\
		& = (-1)^r \, (I-A)^{-1}_r \, (B)_r \nonumber\\
		& \quad \times  \mathcal{H}_1(A-rI, B+rI, C; x, y) \, (C)_r, \quad AB = BA;
		\\[5pt]
		& \left(x^2 \frac{\partial }{\partial x}\right)^r \ [x^{A+(r-1)I} \mathcal{H}_1(A, B;, C; x, \frac{y}{x})]  = x^{A+rI} \, (A)_r \, \mathcal{H}_1(A+rI, B; C; x, \frac{y}{x});
		\\[5pt]
		& \left(x^2 \frac{\partial }{\partial x}\right)^r \ [x^{B+(r-1)I} \mathcal{H}_1(A, B; C; x, {y}{x})]\nonumber\\
		&  = x^{B+rI} \, (B)_r \, \mathcal{H}_1(A, B+rI; C; x, {y}{x}), \quad AB = BA;
		\\[5pt]
		&   \frac{\partial^r }{\partial x^r} \ [ \mathcal{H}_1(A, B; C; x, {y}) x^{C-I}]\nonumber\\
		&  =  (-1)^r \, \mathcal{H}_1(A, B; C-rI; x, {y}) \, x^{C-(r+1)I} \, (I-C)_r.
	\end{align}
\end{theorem}
\begin{theorem}
	Let $A$, $B$, $B'$ and $C$ be matrices in $\mathbb{C}^{r \times r}$. Then the  matrix function $\mathcal{H}_2(A, B, B'; C; x, y)$ satisfies the following differential formulae
	\begin{align}
		&	\frac{\partial ^r}{\partial x^r} \mathcal{H}_2(A, B, B'; C; x, y)\nonumber\\
		& =  (A)_r \, (B)_r \, \mathcal{H}_2(A+rI, B+rI, B'; C+rI; x, y) \,   (C)^{-1}_r,\quad AB = BA;
		\\[5pt]
		&	\frac{\partial ^r}{\partial y^r} \mathcal{H}_2(A, B, B'; C; x, y)\nonumber\\
		& = (-1)^r \, (I-A)^{-1}_r   \mathcal{H}_2(A-rI, B, B'+rI; C; x, y) \,  (B')_r, \quad  CB' = B'C;
		\\[5pt]
		& \left(x^2 \frac{\partial }{\partial x}\right)^r \ [x^{A+(r-1)I} \mathcal{H}_2(A, B, B'; C; x, \frac{y}{x})]\nonumber\\
		&  = x^{A+rI} \, (A)_r \, \mathcal{H}_2(A+rI, B, B'; C; x, \frac{y}{x});
		\\[5pt]
		& \left(x^2 \frac{\partial }{\partial x}\right)^r \ [x^{B+(r-1)I} \mathcal{H}_2(A, B, B'; C; x, {y})]\nonumber\\
		&  = x^{B+rI} \, (B)_r \, \mathcal{H}_2(A, B+rI, B'; C; x, {y}), \quad AB = BA;
		\\[5pt]
		&  \left(y^2 \frac{\partial }{\partial y}\right)^r \ [ \mathcal{H}_2(A, B, B'; C; x, {y}) y^{B'+(r-1)I}]\nonumber\\
		&  =  \mathcal{H}_2(A, B, B' + rI; C; x, {y}) \,  y^{B'+rI} \, (B')_r, \quad CB'= B'C;
		\\[5pt]
		&   \frac{\partial^r }{\partial x^r} \ [ \mathcal{H}_2(A, B, B'; C; x, {y}) x^{C-I}]\nonumber\\
		&  =  (-1)^r \, \mathcal{H}_2(A, B, B'; C-rI; x, {y}) \,  x^{C-(r+1)I} \, (I-C)_r.
	\end{align}
\end{theorem}
\begin{theorem}
	Let $A$, $B$, $C$ be matrices in $\mathbb{C}^{r \times r}$. Then the  matrix function $\mathcal{H}_3(A, B; C; x, y)$ satisfies the following differential formulae
	\begin{align}
		&	\frac{\partial ^r}{\partial x^r} \mathcal{H}_3(A, B; C; x, y)\nonumber\\
		& =  (A)_{r}  \, \mathcal{H}_3(A+rI, B+rI; C+rI; x, y) \,  (B)_r \, (C)^{-1}_r, \quad BC = CB;
		\\[5pt]
		&	\frac{\partial ^r}{\partial y^r} \mathcal{H}_3(A, B; C; x, y)\nonumber\\
		& = (-1)^r \, (I-A)^{-1}_r  \,  \mathcal{H}_3(A-rI, B; C; x, y);
		\\[5pt]
		& \left(x^2 \frac{\partial }{\partial x}\right)^r \ [x^{A+(r-1)I} \mathcal{H}_3(A, B; C; x, \frac{y}{x})]\nonumber\\
		&  = x^{A+rI} \, (A)_r \, \mathcal{H}_3(A+rI, B; C; x, \frac{y}{x});
		\\[5pt]
		&  \left(x^2 \frac{\partial }{\partial x}\right)^r \ [ \mathcal{H}_3(A, B; C; x, {y}) x^{B+(r-1)I}]\nonumber\\
		&  =  \mathcal{H}_3(A, B+rI; C; x, {y}) \,  x^{B+rI} \, (B)_r, \quad BC = CB;
		\\[5pt]
		&   \frac{\partial^r }{\partial x^r} \ [ \mathcal{H}_3(A, B; C; x, {y}) x^{C-I}]  =  (-1)^r \, \mathcal{H}_3(A, B; C-rI; x, {y}) \, x^{C-(r+1)I} \, (I-C)_r.
	\end{align}
\end{theorem}
\begin{theorem}
	Let $A$, $B'$, $C$ be matrices in $\mathbb{C}^{r \times r}$. Then the matrix function $\mathcal{H}_4(A, B'; C; x, y)$ satisfies the following differential formulae
	\begin{align}
		&	\frac{\partial ^r}{\partial x^r} \mathcal{H}_4(A, B'; C; x, y) =  (A)_{r}  \, \mathcal{H}_4(A+rI, B'; C+rI; x, y)  \, (C)^{-1}_r;
		\\[5pt]
		&	\frac{\partial ^r}{\partial y^r} \mathcal{H}_4(A, B'; C; x, y)\nonumber\\
		& = (-1)^r \, (I-A)^{-1}_r  \,  \mathcal{H}_4(A-rI, B'+rI; C; x, y) \, (B')_r, \quad B'C = CB';
		\\[5pt]
		& \left(x^2 \frac{\partial }{\partial x}\right)^r \ [x^{A+(r-1)I} \mathcal{H}_4(A, B'; C; x, \frac{y}{x})]\nonumber\\
		&  = x^{A+rI} \, (A)_r \, \mathcal{H}_4(A+rI, B'; C; x, \frac{y}{x});
		\\[5pt]
		&  \left(y^2 \frac{\partial }{\partial y}\right)^r \ [ \mathcal{H}_4(A, B'; C; x, {y}) y^{B'+(r-1)I}]\nonumber\\
		&  =  \mathcal{H}_4(A, B'+rI; C; x, {y}) \,  y^{B'+rI} \, (B')_r, \quad B'C = CB';
		\\[5pt]
		&   \frac{\partial^r }{\partial x^r} \ [ \mathcal{H}_4(A, B'; C; x, {y}) x^{C-I}]  =  (-1)^r \, \mathcal{H}_4(A, B'; C-rI; x, {y}) \, x^{C-(r+1)I} \, (I-C)_r.
	\end{align}
\end{theorem}
\begin{theorem}
	Let $A$ and $C$ be matrices in $\mathbb{C}^{r \times r}$. Then the matrix function $\mathcal{H}_5(A; C; x, y)$ satisfies the following differential formulae
	\begin{align}
		&	\frac{\partial ^r}{\partial x^r} \mathcal{H}_5(A; C; x, y) =  (A)_{r}  \, \mathcal{H}_5(A+rI; C+rI; x, y)  \, (C)^{-1}_r;
		\\[5pt]
		&	\frac{\partial ^r}{\partial y^r} \mathcal{H}_5(A; C; x, y) = (-1)^r \, (I-A)^{-1}_r  \,  \mathcal{H}_5(A-rI; C; x, y);
		\\[5pt]
		& \left(x^2 \frac{\partial }{\partial x}\right)^r \ [x^{A+(r-1)I} \mathcal{H}_5(A; C; x, \frac{y}{x})] = x^{A+rI} \, (A)_r \, \mathcal{H}_5(A+rI; C; x, \frac{y}{x});
		\\[5pt]
		&   \frac{\partial^r }{\partial x^r} \ [ \mathcal{H}_5(A; C; x, {y}) x^{C-I}]  =  (-1)^r \, \mathcal{H}_5(A; C-rI; x, {y}) \, x^{C-(r+1)I} \, (I-C)_r.
	\end{align}
\end{theorem}
\begin{theorem}
	Let $A$ and $C$ be matrices in $\mathbb{C}^{r \times r}$. Then the matrix function $\mathcal{H}_6(A; C; x, y)$ satisfies the following differential formulae
	\begin{align}
		&	\frac{\partial ^r}{\partial x^r} \mathcal{H}_6(A; C; x, y) =  (A)_{2r}  \, \mathcal{H}_6(A+2rI; C+rI; x, y)  \, (C)^{-1}_r;
		\\[5pt]
		&	\frac{\partial ^r}{\partial y^r} \mathcal{H}_6(A; C; x, y) = (A)_r  \,  \mathcal{H}_6(A+rI; C+rI; x, y) \, (C)_r^{-1};
		\\[5pt]
		& \left(x^2 \frac{\partial }{\partial x}\right)^r \ [x^{A+(r-1)I} \mathcal{H}_6(A; C; x^2, {y}{x})] = x^{A+rI} \, (A)_r \, \mathcal{H}_6(A+rI; C; x^2, {y}{x});
		\\[5pt]
		&   \frac{\partial^r }{\partial x^r} \ [ \mathcal{H}_6(A; C; x, {x y}) x^{C-I}]  =  (-1)^r \, \mathcal{H}_6(A; C-rI; x, {x y}) \, x^{C-(r+1)I} \, (I-C)_r.
	\end{align}
\end{theorem}
\begin{theorem}
	Let $A$, $C$ and $C'$ be matrices in $\mathbb{C}^{r \times r}$. Then the matrix function $\mathcal{H}_7(A; C, C'; x, y)$ satisfies the following differential formulae
	\begin{align}
		&	\frac{\partial ^r}{\partial x^r} \mathcal{H}_7(A; C, C'; x, y)\nonumber\\
		& =  (A)_{2r}  \, \mathcal{H}_7(A+2rI; C+rI, C'; x, y)  \, (C)^{-1}_r, \quad CC' = C'C;
		\\[5pt]
		&	\frac{\partial ^r}{\partial y^r} \mathcal{H}_7(A; C, C'; x, y) = (A)_r  \,  \mathcal{H}_7(A+rI; C, C'+rI; x, y) \, (C')_r^{-1};
		\\[5pt]
		& \left(x^2 \frac{\partial }{\partial x}\right)^r \ [x^{A+(r-1)I} \mathcal{H}_7(A; C, C'; x^2, {y}{x})]\nonumber\\
		& = x^{A+rI} \, (A)_r \, \mathcal{H}_7(A+rI; C, C'; x^2, {y}{x});
		\\[5pt]
		&   \frac{\partial^r }{\partial x^r} \ [ \mathcal{H}_7(A; C, C'; x, { y}) x^{C-I}]  \nonumber\\
		&=  (-1)^r \, \mathcal{H}_7(A; C-rI, C'; x, {y}) \, x^{C-(r+1)I} \, (I-C)_r, \quad CC' = C'C
		\\[5pt]
		&   \frac{\partial^r }{\partial y^r} \ [ \mathcal{H}_7(A; C, C'; x, { y}) y^{C'-I}]\nonumber\\
		&  =  (-1)^r \, \mathcal{H}_7(A; C, C'-rI; x, {y}) \, x^{C'-(r+1)I} \, (I-C')_r.
	\end{align}
\end{theorem}
\begin{theorem}
	Let $A$ and $B$ be matrices in $\mathbb{C}^{r \times r}$. Then the matrix function $\mathcal{H}_8(A, B; x, y)$ satisfies the following differential formulae
	\begin{align}
		&	\frac{\partial ^r}{\partial x^r} \mathcal{H}_8(A, B; x, y) =  (-1)^r \, (A)_{2r}  \, \mathcal{H}_8(A+2rI, B-rI; x, y)  \, (I-B)^{-1}_r;
		\\[5pt]
		&	\frac{\partial ^r}{\partial y^r} \mathcal{H}_8(A, B; x, y) = (-1)^r \, (I-A)^{-1}_r  \,  \mathcal{H}_8(A-rI, B+rI; x, y) \, (B)_r;
		\\[5pt]
		& \left(x^2 \frac{\partial }{\partial x}\right)^r \ [x^{A+(r-1)I} \mathcal{H}_8(A, B; x^2, \frac{y}{x})] = x^{A+rI} \, (A)_r \, \mathcal{H}_8(A+rI, B; x^2, \frac{y}{x});
		\\[5pt]
		& \left(y^2 \frac{\partial }{\partial y}\right)^r \ [y^{B+(r-1)I} \mathcal{H}_8(A, B; \frac{x}{y}, {y})] = \mathcal{H}_8(A, B+rI; \frac{x}{y}, {y}) \, y^{B+rI} \, (B)_r.
	\end{align}
\end{theorem}
\begin{theorem}
	Let $A$, $B$ and $C$ be matrices in $\mathbb{C}^{r \times r}$. Then the matrix function $\mathcal{H}_9(A, B; C; x, y)$ satisfies the following differential formulae
	\begin{align}
		&	\frac{\partial ^r}{\partial x^r} \mathcal{H}_9(A, B; C; x, y) =  (A)_{2r}  \, \mathcal{H}_9(A+2rI, B; C+rI; x, y)  \, (C)^{-1}_r;
		\\[5pt]
		&	\frac{\partial ^r}{\partial y^r} \mathcal{H}_9(A, B; C; x, y) = (-1)^r \, (I-A)^{-1}_r  \,  \mathcal{H}_9(A-rI, B+rI; C; x, y) \, (B)_r, \quad BC = CB;
		\\[5pt]
		& \left(x^2 \frac{\partial }{\partial x}\right)^r \ [x^{A+(r-1)I} \mathcal{H}_9(A, B; C; x^2, \frac{y}{x})] = x^{A+rI} \, (A)_r \, \mathcal{H}_9(A+rI, B; C; x^2, \frac{y}{x});
		\\[5pt]
		& \left(y^2 \frac{\partial }{\partial y}\right)^r \ [y^{B+(r-1)I} \mathcal{H}_9(A, B; C; {x}, {y})] \nonumber\\
		& = \mathcal{H}_9(A, B+rI; C; {x}, {y}) \, y^{B+rI} \, (B)_r, \quad BC = CB;
		\\[5pt]
		&  \frac{\partial^r }{\partial x^r} \ [ \mathcal{H}_9(A, B; C; {x}, {y}) x^{C-I}] = (-1)^r \, \mathcal{H}_9(A, B; C-rI; {x}, {y}) \, x^{C-(r+1)I} \, (1-C)_r.
	\end{align}
\end{theorem}
\begin{theorem}
	Let $A$ and $C$ be matrices in $\mathbb{C}^{r \times r}$. Then the matrix function $\mathcal{H}_{10}(A; C; x, y)$ satisfies the following differential formulae
	\begin{align}
		&	\frac{\partial ^r}{\partial x^r} \mathcal{H}_{10}(A; C; x, y) =   (A)_{2r}  \, \mathcal{H}_{10}(A+2rI; C+rI; x, y)  \, (C)^{-1}_r;
		\\[5pt]
		&	\frac{\partial ^r}{\partial y^r} \mathcal{H}_{10}(A; C; x, y) = (-1)^r \, (I-A)^{-1}_r  \,  \mathcal{H}_{10}(A-rI; C; x, y);
		\\[5pt]
		& \left(x^2 \frac{\partial }{\partial x}\right)^r \ [x^{A+(r-1)I} \mathcal{H}_{10}(A; C; x^2, \frac{y}{x})] = x^{A+rI} \, (A)_r \, \mathcal{H}_{10}(A+rI; C; x^2, \frac{y}{x});
		\\[5pt]
		&  \frac{\partial^r }{\partial x^r} \ [ \mathcal{H}_{10}(A; C; {x}, {y}) x^{C-I}] = (-1)^r \, \mathcal{H}_{10}(A; C-rI; {x}, {y}) \, x^{C-(r+1)I} \, (1-C)_r.
	\end{align}
\end{theorem}
\begin{theorem}
	Let $A$, $B$, $C$ and $C'$ be matrices in $\mathbb{C}^{r \times r}$. Then the  matrix function $\mathcal{H}_{11}(A, B, C; C'; x, y)$ satisfies the following differential formulae
	\begin{align}
		&	\frac{\partial ^r}{\partial x^r} \mathcal{H}_{11}(A, B, C; C'; x, y) =   (A)_{2}  \, \mathcal{H}_{11}(A+rI, B, C; C'+rI; x, y)  \, (C')^{-1}_r;
		\\[5pt]
		&	\frac{\partial ^r}{\partial y^r} \mathcal{H}_{11}(A, B, C; C'; x, y)\nonumber\\
		& = (-1)^r \, (I-A)^{-1}_r  \, (B)_r\nonumber\\
		& \quad \times  \mathcal{H}_{11}(A-rI, B+rI, C+rI; C'; x, y) (C)_r, \quad AB = BA, CC' = C'C;
		\\[5pt]
		& \left(x^2 \frac{\partial }{\partial x}\right)^r \ [x^{A+(r-1)I} \mathcal{H}_{11}(A, B, C; C'; x, \frac{y}{x})] \nonumber\\
		& = x^{A+rI} \, (A)_r \, \mathcal{H}_{11}(A+rI, B, C; C'; x, \frac{y}{x});
		\\[5pt]
		& \left(y^2 \frac{\partial }{\partial y}\right)^r \ [y^{B+(r-1)I} \mathcal{H}_{11}(A, B, C; C'; x, {y})] \nonumber\\
		& = y^{B+rI} \, (B)_r \, \mathcal{H}_{11}(A, B+rI, C; C'; x, {y});
		\\[5pt]
		&  \frac{\partial^r }{\partial x^r} \ [ \mathcal{H}_{11}(A, B, C; C'; {x}, {y}) x^{C'-I}] \nonumber\\
		& = (-1)^r \, \mathcal{H}_{11}(A, B, C; C'-rI; {x}, {y}) \, x^{C'-(r+1)I} \, (1-C')_r.
	\end{align}
\end{theorem}

\section{Infinite summation formulae}
This section deals with the study of infinite summation formulae satisfied by Horn matrix functions and their confluent forms. In the theorem given below, we give the infinite summation formulae satisfied by first Horn matrix function ${G}_{1}(A, B, B'; x, y)$.
\begin{theorem}
	Let $A$, $B$ and $B'$ be matrices in $\mathbb{C}^{r \times r}$. Then the Horn matrix function ${G}_{1}(A, B, B'; x, y)$ satisfies the following infinite summation formulae
	\begin{align}
		& (1-t)^{-A} G_{1} \left(A, B, B'; \frac{x}{1-t}, \frac{y}{1-t}\right) \nonumber\\
		& = \sum_{n = 0}^{\infty} \frac{(A)_n}{n!} G_{1} (A+nI, B, B'; {x}, y) t^n, \quad  \vert t\vert < 1;\label{a5.1}
		\\[5pt]
		&  G_{1} \left(A, B, B'; {x}{(1-t)}, \frac{y}{1-t}\right) \, (1-t)^{-B} \nonumber\\
		& = \sum_{n = 0}^{\infty}  G_{1} (A, B+nI, B'; {x}, y) \, \frac{(B)_n}{n!} \,  t^n, \quad BB' = B'B,   \vert t\vert < 1;
		\\[5pt]
		&  G_{1} \left(A, B, B'; \frac{x}{1-t}, {y}{(1-t)}\right) (1-t)^{-B'} \nonumber\\
		& = \sum_{n = 0}^{\infty}  G_{1} (A, B, B'+nI; {x}, y) \, \frac{(B')_n}{n!} \,t^n, \quad  \vert t\vert < 1.
	\end{align}
\end{theorem}
\begin{proof}
	From the definition of Horn's matrix function ${G}_{1}(A, B, B'; x, y)$, we have
	\begin{align}
		& (1-t)^{-A} G_{1} \left(A, B, B'; \frac{x}{1-t}, \frac{y}{1-t}\right)\nonumber\\
		& = \sum_{l, m= 0}^{\infty} (1-t)^{-(A +(l+m)I)} (A)_{l+m} (B)_{m-l} (B')_{l-m} \frac{x^l \, y^m}{l! \, m!}. 
	\end{align}
Using the matrix identities, for $\vert t \vert < 1$,  $(1-t)^{-(A +(l+m)I)} = \sum_{n =0}^{\infty} \frac{(A +(l+m)I)_n}{n!} t^n$ and $(A +(l+m)I)_n (A)_{l+m} = (A)_n (A+nI)_{l+m}$, we get
\begin{align}
	& (1-t)^{-A} G_{1} \left(A, B, B'; \frac{x}{1-t}, \frac{y}{1-t}\right)\nonumber\\
	& = \sum_{l, m, n= 0}^{\infty}  \frac{(A)_n}{n!} (A+nI)_{l+m} (B)_{m-l} (B')_{l-m} \frac{x^l \, y^m}{l! \, m!} \, t^n\nonumber\\
	& = \sum_{n = 0}^{\infty} \frac{(A)_n}{n!} G_{1} (A+nI, B, B'; {x}, y) t^n, \quad  \vert t\vert < 1. 
\end{align}
It completes the proof of \eqref{a5.1}. Similarly, the other infinite summation can be proved.
\end{proof}
\begin{theorem}
	Let $A$, $A'$, $B$ and $B'$ be matrices in $\mathbb{C}^{r \times r}$. Then the Horn matrix function ${G}_{2}(A, A', B, B'; x, y)$ satisfies the following infinite summation formulae
	\begin{align}
		& (1-t)^{-A} G_{2} \left(A, A', B, B'; \frac{x}{1-t}, {y}\right) \nonumber\\
		& = \sum_{n = 0}^{\infty} \frac{(A)_n}{n!} G_{2} (A+nI, A', B, B'; {x}, y) t^n, \quad  \vert t\vert < 1;
		\\[5pt]
		& (1-t)^{-A'} G_{2} \left(A, A', B, B'; {x}, \frac{y}{1-t}\right) \nonumber\\
		& = \sum_{n = 0}^{\infty} \frac{(A')_n}{n!} G_{2} (A, A'+nI, B, B'; {x}, y) t^n, \quad AA' = A'A, \vert t\vert < 1;
		\\[5pt]
		&  G_{2} \left(A, A', B, B'; {x}{(1-t)}, \frac{y}{1-t}\right) \, (1-t)^{-B} \nonumber\\
		& = \sum_{n = 0}^{\infty}  G_{2} (A, A', B+nI, B'; {x}, y) \, \frac{(B)_n}{n!} \,  t^n, \quad BB' = B'B,   \vert t\vert < 1;
		\\[5pt]
		&  G_{2} \left(A, A', B, B'; \frac{x}{1-t}, {y}{(1-t)}\right) (1-t)^{-B'} \nonumber\\
		& = \sum_{n = 0}^{\infty}  G_{2} (A, A', B, B'+nI; {x}, y) \, \frac{(B')_n}{n!} \,t^n, \quad  \vert t\vert < 1.
	\end{align}
\end{theorem}
\begin{theorem}
	Let $A$ and $A'$ be matrices in $\mathbb{C}^{r \times r}$. Then the Horn matrix function ${G}_{3}(A, A'; x, y)$ satisfies the following infinite summation formulae
	\begin{align}
		& (1-t)^{-A} G_{3} \left(A, A'; {x}{(1-t)}, \frac{y}{(1-t)^2}\right) \nonumber\\
		& = \sum_{n = 0}^{\infty} \frac{(A)_n}{n!} G_{3} (A+nI, A'; {x}, y) t^n, \quad  \vert t\vert < 1;
		\\[5pt]
		&  G_{3} \left(A, A'; \frac{x}{(1-t)^2}, {y}{(1-t)}\right) \, (1-t)^{-A'} \nonumber\\
		& = \sum_{n = 0}^{\infty}  G_{3} (A, A'+nI; {x}, y) \, \frac{(A')_n}{n!} \,  t^n, \quad   \vert t\vert < 1.
	\end{align}
\end{theorem}
\begin{theorem}
	Let $A$, $B$, $C$ and $C'$ be matrices in $\mathbb{C}^{r \times r}$. Then the Horn matrix function ${H}_{1}(A, B, C; C'; x, y)$ satisfies the following infinite summation formulae
	\begin{align}
& (1-t)^{-A} H_{1} \left(A, B, C; C'; \frac{x}{1-t}, y(1-t)\right) \nonumber\\
& = \sum_{n = 0}^{\infty} \frac{(A)_n}{n!} H_{1} (A+nI, B, C; C'; {x}, y) t^n, \quad  \vert t\vert < 1;
\\[5pt]
	& (1-t)^{-B} H_{1} \left(A, B, C; C'; \frac{x}{1-t}, \frac{y}{1-t}\right) \nonumber\\
	& = \sum_{n = 0}^{\infty} \frac{(B)_n}{n!} H_{1} (A, B+nI, C; C'; {x}, y) t^n, \quad AB = BA,   \vert t\vert < 1;
	\\[5pt]
	&  H_{1} \left(A, B, C; C'; {x}, \frac{y}{1-t}\right) (1-t)^{-C} \nonumber\\
	& = \sum_{n = 0}^{\infty}  H_{1} (A, B, C+nI; C'; {x}, y) \, \frac{(C)_n}{n!} \,t^n, \quad CC' = C'C,   \vert t\vert < 1.
\end{align}
\end{theorem}
\begin{theorem}
	Let $A$, $B$, $C$, $C'$ and $C''$ be matrices in $\mathbb{C}^{r \times r}$. Then the Horn matrix function ${H}_{2}(A, B, C, C'; C''; x, y)$ satisfies the following infinite summation formulae
	\begin{align}
			& (1-t)^{-A} H_{2} \left(A, B, C, C'; C''; \frac{x}{1-t}, {y}{(1-t)}\right)  \nonumber\\
		& = \sum_{n = 0}^{\infty} \frac{(A)_n}{n!} H_{2} (A+nI, B, C, C'; C''; {x}, y) \, t^n, \quad  \vert t\vert < 1;
		\\[5pt]
		& (1-t)^{-B} H_{2} \left(A, B, C, C'; C''; \frac{x}{1-t}, {y}\right)  \nonumber\\
	& = \sum_{n = 0}^{\infty} \frac{(B)_n}{n!} H_{2} (A, B+nI, C, C'; C''; {x}, y) \, t^n, \quad AB = BA, \vert t\vert < 1;
	\\[5pt]
		&  H_{2} \left(A, B, C, C'; C''; {x}, \frac{y}{1-t}\right) (1-t)^{-C} \nonumber\\
		& = \sum_{n = 0}^{\infty}  H_{2} (A, B, C+nI, C'; C''; {x}, y) \, \frac{(C)_n}{n!} \, t^n, \  CC' = C'C, CC'' = C''C,  \vert t\vert < 1;
		\\[5pt]
		&  H_{2} \left(A, B, C, C'; C''; {x}, \frac{y}{1-t}\right) \, (1-t)^{-C'} \nonumber\\
		& = \sum_{n = 0}^{\infty} H_{2} (A, B, C, C'+nI; C''; {x}, y) \,  \frac{(B)_n}{n!} \, t^n, \quad CC''=C''C,   \vert t\vert < 1.
	\end{align}
\end{theorem}
\begin{theorem}
	Let $A$, $B$ and $C$ be matrices in $\mathbb{C}^{r \times r}$. Then the Horn matrix function ${H}_{3}(A, B; C; x, y)$ satisfies the following infinite summation formulae
	\begin{align}
		& (1-t)^{-A} H_{3} \left(A, B; C; \frac{x}{(1-t)^2}, \frac{y}{1-t}\right) \nonumber\\
		& = \sum_{n = 0}^{\infty} \frac{(A)_n}{n!} H_{3} (A+nI, B; C; {x}, y) t^n, \quad  \vert t\vert < 1;
		\\[5pt]
		& (1-t)^{-B} H_{3} \left(A, B; C; {x}, \frac{y}{1-t}\right) \nonumber\\
		& = \sum_{n = 0}^{\infty} \frac{(B)_n}{n!} H_{3} (A, B+nI; C; {x}, y) t^n, \quad AB = BA,   \vert t\vert < 1.
	\end{align}
\end{theorem}
\begin{theorem}
	Let $A$, $B$, $C$ and $C'$ be matrices in $\mathbb{C}^{r \times r}$. Then the Horn matrix function ${H}_{4}(A, B; C, C'; x, y)$ satisfies the following infinite summation formulae
	\begin{align}
		& (1-t)^{-A} H_{4} \left(A, B; C, C'; \frac{x}{(1-t)^2}, \frac{y}{1-t}\right) \nonumber\\
		& = \sum_{n = 0}^{\infty} \frac{(A)_n}{n!} H_{4} (A+nI, B; C, C'; {x}, y) t^n, \quad  \vert t\vert < 1;
		\\[5pt]
		& (1-t)^{-B} H_{4} \left(A, B; C, C'; {x}, \frac{y}{1-t}\right) \nonumber\\
		& = \sum_{n = 0}^{\infty} \frac{(B)_n}{n!} H_{4} (A, B+nI; C, C'; {x}, y) t^n, \quad AB = BA,   \vert t\vert < 1.
	\end{align}
\end{theorem}
\begin{theorem}
	Let $A$, $B$ and $C$ be matrices in $\mathbb{C}^{r \times r}$. Then the Horn matrix function ${H}_{5}(A, B; C; x, y)$ satisfies the following infinite summation formulae
	\begin{align}
		& (1-t)^{-A} H_{5} \left(A, B; C; \frac{x}{(1-t)^2}, \frac{y}{1-t}\right) \nonumber\\
		& = \sum_{n = 0}^{\infty} \frac{(A)_n}{n!} H_{5} (A+nI, B; C; {x}, y) t^n, \quad  \vert t\vert < 1;
		\\[5pt]
		& (1-t)^{-B} H_{5} \left(A, B; C; {x}(1-t), \frac{y}{1-t}\right) \nonumber\\
		& = \sum_{n = 0}^{\infty} \frac{(B)_n}{n!} H_{5} (A, B+nI; C; {x}, y) t^n, \quad AB = BA,   \vert t\vert < 1.
	\end{align}
\end{theorem}
\begin{theorem}
	Let $A$, $B$ and $C$ be matrices in $\mathbb{C}^{r \times r}$. Then the Horn matrix function ${H}_{6}(A, B; C; x, y)$ satisfies the following infinite summation formulae
	\begin{align}
		& (1-t)^{-A} H_{6} \left(A, B; C; \frac{x}{(1-t)^2}, {y}{(1-t)}\right) \nonumber\\
		& = \sum_{n = 0}^{\infty} \frac{(A)_n}{n!} H_{6} (A+nI, B; C; {x}, y) t^n, \quad  \vert t\vert < 1;
		\\[5pt]
		& (1-t)^{-B} H_{6} \left(A, B; C; {x}(1-t), \frac{y}{1-t}\right) \nonumber\\
		& = \sum_{n = 0}^{\infty} \frac{(B)_n}{n!} H_{3} (A, B+nI; C; {x}, y) t^n, \quad AB = BA,   \vert t\vert < 1.
	\end{align}
\end{theorem}
\begin{theorem}
	Let $A$, $B$, $C$ and $C'$ be matrices in $\mathbb{C}^{r \times r}$. Then the Horn matrix function ${H}_{7}(A, B; C, C'; x, y)$ satisfies the following infinite summation formulae
	\begin{align}
		& (1-t)^{-A} H_{7} \left(A, B; C, C'; \frac{x}{(1-t)^2}, {y}{(1-t)}\right) \nonumber\\
		& = \sum_{n = 0}^{\infty} \frac{(A)_n}{n!} H_{7} (A+nI, B; C, C'; {x}, y) t^n, \quad  \vert t\vert < 1;
		\\[5pt]
		& (1-t)^{-B} H_{7} \left(A, B; C, C'; {x}, \frac{y}{1-t}\right) \nonumber\\
		& = \sum_{n = 0}^{\infty} \frac{(B)_n}{n!} H_{7} (A, B+nI; C, C'; {x}, y) t^n, \quad AB = BA,   \vert t\vert < 1
		\\[5pt]
		&  H_{7} \left(A, B; C, C'; {x}, \frac{y}{1-t}\right) \, (1-t)^{-C} \nonumber\\
		& = \sum_{n = 0}^{\infty}  H_{7} (A, B+nI; C, C'; {x}, y) \, \frac{(C)_n}{n!} t^n, \quad   \vert t\vert < 1.
	\end{align}
\end{theorem}
\begin{theorem}
	Let $A$, $B$ and $B'$ be matrices in $\mathbb{C}^{r \times r}$. Then the  matrix function ${\Gamma}_{1}(A, B, B'; x, y)$ satisfies the following infinite summation formulae
	\begin{align}
		& (1-t)^{-A} \Gamma_{1} \left(A, B, B'; \frac{x}{1-t}, {y}\right) \nonumber\\
		& = \sum_{n = 0}^{\infty} \frac{(A)_n}{n!} \Gamma_{1} (A+nI, B, B'; {x}, y) t^n, \quad  \vert t\vert < 1;
		\\[5pt]
		&  \Gamma_{1} \left(A, B, B'; {x}{(1-t)}, \frac{y}{1-t}\right) \, (1-t)^{-B} \nonumber\\
		& = \sum_{n = 0}^{\infty}  \Gamma_{1} (A, B+nI, B'; {x}, y) \, \frac{(B)_n}{n!} \,  t^n, \quad BB' = B'B,   \vert t\vert < 1;
		\\[5pt]
		&  \Gamma_{1} \left(A, B, B'; \frac{x}{1-t}, {y}{(1-t)}\right) (1-t)^{-B'} \nonumber\\
		& = \sum_{n = 0}^{\infty}  \Gamma_{1} (A, B, B'+nI; {x}, y) \, \frac{(B')_n}{n!} \,t^n, \quad  \vert t\vert < 1.
	\end{align}
\end{theorem}
\begin{theorem}
	Let $B$ and $B'$ be matrices in $\mathbb{C}^{r \times r}$. Then the  matrix function ${\Gamma}_{2}(B, B'; x, y)$ satisfies the following infinite summation formulae
	\begin{align}
		&  (1-t)^{-B} \, \Gamma_{2} \left(B, B'; {x}{(1-t)}, \frac{y}{1-t}\right)\nonumber\\
		&  = \sum_{n = 0}^{\infty}  \frac{(B)_n}{n!} \, \Gamma_{2} (B+nI, B'; {x}, y) \,   t^n, \quad   \vert t\vert < 1;
		\\[5pt]
		&  \Gamma_{2} \left(B, B'; \frac{x}{1-t}, {y}{(1-t)}\right) (1-t)^{-B'} \nonumber\\
		& = \sum_{n = 0}^{\infty}  \Gamma_{2} (B, B'+nI; {x}, y) \, \frac{(B')_n}{n!} \,t^n, \quad  \vert t\vert < 1.
	\end{align}
\end{theorem}
\begin{theorem}
	Let $A$, $B$ and $C$ be matrices in $\mathbb{C}^{r \times r}$. Then the matrix function $\mathcal{H}_1(A, B; C; x, y)$ satisfies the following infinite summation formulae
	\begin{align}
	&  (1-t)^{-A} \, \mathcal{H}_1 \left(A, B; C; \frac{x}{1-t}, y{(1-t)}\right)\nonumber\\
	&  = \sum_{n = 0}^{\infty}  \frac{(A)_n}{n!} \, \mathcal{H}_1 \left(A+nI, B; C; {x}, y\right) \,   t^n, \quad   \vert t\vert < 1;
	\\[5pt]
	&  (1-t)^{-B} \, \mathcal{H}_1 \left(A, B; C; \frac{x}{1-t}, \frac{y}{1-t}\right)\nonumber\\
&  = \sum_{n = 0}^{\infty}  \frac{(B)_n}{n!} \, \mathcal{H}_1 \left(A, B+nI; C; {x}, y\right) \,   t^n, \quad  AB = BA,  \vert t\vert < 1.
\end{align}
\end{theorem}
\begin{theorem}
	Let $A$, $B$, $B'$ and $C$ be matrices in $\mathbb{C}^{r \times r}$. Then the matrix function $\mathcal{H}_2(A, B, B'; C; x, y)$ satisfies the following infinite summation formulae
	\begin{align}
		&  (1-t)^{-A} \, \mathcal{H}_2 \left(A, B, B'; C; \frac{x}{1-t}, y{(1-t)}\right)\nonumber\\
		&  = \sum_{n = 0}^{\infty}  \frac{(A)_n}{n!} \, \mathcal{H}_2 \left(A+nI, B, B'; C; {x}, y\right) \,   t^n, \quad   \vert t\vert < 1;
		\\[5pt]
		&  (1-t)^{-B} \, \mathcal{H}_2 \left(A, B, B'; C; \frac{x}{1-t}, {y}\right)\nonumber\\
		&  = \sum_{n = 0}^{\infty}  \frac{(B)_n}{n!} \, \mathcal{H}_2 \left(A, B+nI, B'; C; {x}, y\right) \,   t^n, \quad  AB = BA,  \vert t\vert < 1
			\\[5pt]
		&  \mathcal{H}_2 \left(A, B, B'; C; {x}, \frac{y}{1-t}\right) \, (1-t)^{-B'}\nonumber\\
		&  = \sum_{n = 0}^{\infty}   \mathcal{H}_2 \left(A, B, B'+nI; C; {x}, y\right) \, \frac{(B')_n}{n!} \,  t^n, \quad  B'C = CB',  \vert t\vert < 1.
	\end{align}
\end{theorem}
\begin{theorem}
	Let $A$, $B$ and $C$ be matrices in $\mathbb{C}^{r \times r}$. Then the matrix function $\mathcal{H}_3(A, B; C; x, y)$ satisfies the following infinite summation formulae
	\begin{align}
		&  (1-t)^{-A} \, \mathcal{H}_3 \left(A, B; C; \frac{x}{1-t}, y{(1-t)}\right)\nonumber\\
		&  = \sum_{n = 0}^{\infty}  \frac{(A)_n}{n!} \, \mathcal{H}_3 \left(A+nI, B; C; {x}, y\right) \,   t^n, \quad   \vert t\vert < 1;
		\\[5pt]
		&  (1-t)^{-B} \, \mathcal{H}_3 \left(A, B; C; \frac{x}{1-t}, {y}\right)\nonumber\\
		&  = \sum_{n = 0}^{\infty}  \frac{(B)_n}{n!} \, \mathcal{H}_3 \left(A, B+nI; C; {x}, y\right) \,   t^n, \quad  AB = BA,  \vert t\vert < 1.
	\end{align}
\end{theorem}
\begin{theorem}
	Let $A$, $B'$ and $C$ be matrices in $\mathbb{C}^{r \times r}$. Then the matrix function $\mathcal{H}_4(A, B'; C; x, y)$ satisfies the following infinite summation formulae
	\begin{align}
		&  (1-t)^{-A} \, \mathcal{H}_4 \left(A, B'; C; \frac{x}{1-t}, y{(1-t)}\right)\nonumber\\
		&  = \sum_{n = 0}^{\infty}  \frac{(A)_n}{n!} \, \mathcal{H}_4 \left(A+nI, B'; C; {x}, y\right) \,   t^n, \quad   \vert t\vert < 1;
		\\[5pt]
		&  (1-t)^{-B'} \, \mathcal{H}_4 \left(A, B'; C; {x},  \frac{y}{1-t}\right)\nonumber\\
		&  = \sum_{n = 0}^{\infty}  \frac{(B')_n}{n!} \, \mathcal{H}_4 \left(A, B'+nI; C; {x}, y\right) \,   t^n, \quad  AB' = B'A,  \vert t\vert < 1.
	\end{align}
\end{theorem}
\begin{theorem}
	Let $A$ and $C$ be matrices in $\mathbb{C}^{r \times r}$. Then the matrix function $\mathcal{H}_5(A; C; x, y)$ satisfies the following infinite summation formulae
	\begin{align}
		&  (1-t)^{-A} \, \mathcal{H}_5 \left(A; C; \frac{x}{1-t}, y{(1-t)}\right) = \sum_{n = 0}^{\infty}  \frac{(A)_n}{n!} \, \mathcal{H}_5 \left(A+nI; C; {x}, y\right) \,   t^n, \quad   \vert t\vert < 1.
	\end{align}
\end{theorem}
\begin{theorem}
	Let $A$ and $C$ be matrices in $\mathbb{C}^{r \times r}$. Then the matrix function $\mathcal{H}_6(A; C; x, y)$ satisfies the following infinite summation formulae
	\begin{align}
		&  (1-t)^{-A} \, \mathcal{H}_6 \left(A; C; \frac{x}{(1-t)^2}, \frac{y}{1-t}\right) = \sum_{n = 0}^{\infty}  \frac{(A)_n}{n!} \, \mathcal{H}_6 \left(A+nI; C; {x}, y\right) \,   t^n, \quad   \vert t\vert < 1.
	\end{align}
\end{theorem}
\begin{theorem}
	Let $A$, $C$ and $C'$ be matrices in $\mathbb{C}^{r \times r}$. Then the matrix function $\mathcal{H}_7(A; C, C'; x, y)$ satisfies the following infinite summation formulae
	\begin{align}
		&  (1-t)^{-A} \, \mathcal{H}_7\left(A; C, C'; \frac{x}{(1-t)^2}, \frac{y}{1-t}\right)\nonumber\\
		&  = \sum_{n = 0}^{\infty}  \frac{(A)_n}{n!} \, \mathcal{H}_7\left(A+nI; C, C'; {x}, y\right) \,   t^n, \quad   \vert t\vert < 1.
	\end{align}
\end{theorem}
\begin{theorem}
	Let $A$ and $B$ be matrices in $\mathbb{C}^{r \times r}$. Then the matrix function $\mathcal{H}_8(A, B; x, y)$ satisfies the following infinite summation formulae
	\begin{align}
		&  (1-t)^{-A} \, \mathcal{H}_8 \left(A, B; \frac{x}{(1-t)^2}, y{(1-t)}\right) = \sum_{n = 0}^{\infty}  \frac{(A)_n}{n!} \, \mathcal{H}_8 \left(A+nI, B; {x}, y\right) \,   t^n, \quad   \vert t\vert < 1;
		\\[5pt]
		&   \mathcal{H}_8 \left(A, B; {x}{(1-t)}, \frac{y}{1-t}\right) (1-t)^{-B}  = \sum_{n = 0}^{\infty}   \mathcal{H}_8 \left(A, B+nI; {x}, y\right) \, \frac{(B)_n}{n!} \,  t^n, \quad   \vert t\vert < 1.
		\end{align}.
\end{theorem}
\begin{theorem}
	Let $A$, $B$ and $C$ be matrices in $\mathbb{C}^{r \times r}$. Then the matrix function $\mathcal{H}_9(A, B; C; x, y)$ satisfies the following infinite summation formulae
	\begin{align}
		&  (1-t)^{-A} \, \mathcal{H}_9 \left(A, B; C; \frac{x}{(1-t)^2}, y{(1-t)}\right)\nonumber\\
		&  = \sum_{n = 0}^{\infty}  \frac{(A)_n}{n!} \, \mathcal{H}_9 \left(A+nI, B; C; {x}, y\right) \,   t^n, \quad   \vert t\vert < 1;
		\\[5pt]
		&  (1-t)^{-B} \, \mathcal{H}_9 \left(A, B; C; {x}, \frac{y}{1-t}\right)\nonumber\\
		&  = \sum_{n = 0}^{\infty}  \frac{(B)_n}{n!} \, \mathcal{H}_9 \left(A, B+nI; C; {x}, y\right) \,   t^n, \quad  AB = BA,  \vert t\vert < 1.
	\end{align}
\end{theorem}
\begin{theorem}
	Let $A$ and $C$ be matrices in $\mathbb{C}^{r \times r}$. Then the matrix function $\mathcal{H}_{10}(A; C; x, y)$ satisfies the following infinite summation formulae
	\begin{align}
		&  (1-t)^{-A} \, \mathcal{H}_{10} \left(A; C; \frac{x}{(1-t)^2}, y{(1-t)}\right) = \sum_{n = 0}^{\infty}  \frac{(A)_n}{n!} \, \mathcal{H}_{10} \left(A+nI; C; {x}, y\right) \,   t^n, \quad   \vert t\vert < 1.
	\end{align}
\end{theorem}
\begin{theorem}
	Let $A$, $B$, $C$ and $C'$ be matrices in $\mathbb{C}^{r \times r}$. Then the matrix function $\mathcal{H}_{11}(A, B, C; C'; x, y)$ satisfies the following infinite summation formulae
	\begin{align}
		&  (1-t)^{-A} \, \mathcal{H}_{11} \left(A, B, C; C'; \frac{x}{1-t}, y{(1-t)}\right)\nonumber\\
		&  = \sum_{n = 0}^{\infty}  \frac{(A)_n}{n!} \, \mathcal{H}_{11} \left(A+nI, B, C; C'; {x}, y\right) \,   t^n, \quad   \vert t\vert < 1;
		\\[5pt]
		&  (1-t)^{-B} \, \mathcal{H}_{11} \left(A, B, C; C'; {x}, \frac{y}{1-t}\right)\nonumber\\
		&  = \sum_{n = 0}^{\infty}  \frac{(B)_n}{n!} \, \mathcal{H}_{11} \left(A, B+nI, C; C'; {x}, y\right) \,   t^n, \quad  AB = BA,  \vert t\vert < 1
		\\[5pt]
		&   \mathcal{H}_{11} \left(A, B, C; C'; {x}, \frac{y}{1-t}\right)\, (1-t)^{-C} \nonumber\\
		&  = \sum_{n = 0}^{\infty}   \mathcal{H}_{11} \left(A, B, C+nI; C'; {x}, y\right) \,  \frac{(C)_n}{n!} \, t^n, \quad  CC' = C'C,  \vert t\vert < 1.
	\end{align}
\end{theorem}
	
	\section{Conclusion}
 In this paper, we studied the Horn functions and its confluent cases with the matrices as parameters. We discuss the regions of convergence and give the system of partial differential equations of bilateral type satisfy by these matrix functions.  We also determine certain integral representation of these matrix functions. In last, we give the differential formulae and infinite summation formulae induced from these matrix functions. These matrix functions will enrich the literature in theory of special functions and are capable  to find the new applications in mathematics as well as in physics.

\end{document}